\documentclass[11pt]{amsart}
\usepackage{etex} %% This is for an extension and allows us to use many packages.

\usepackage{amsmath}
\usepackage{amssymb}
\usepackage{amsthm}
\usepackage{amsfonts}
\usepackage[usenames,dvipsnames,svgnames]{xcolor}
\usepackage{hyperref}
\hypersetup{colorlinks=true,
allcolors=blue!75!black}
\usepackage{mathrsfs}
\usepackage{pictexwd, dcpic}
\usepackage{pinlabel}
\usepackage{stmaryrd}

\usepackage{tikz}
\usepackage{tikz-cd}
\usepackage{caption}
\usetikzlibrary{decorations.markings}

\newcommand{\QQ}{\mathbb{Q}}
\newcommand{\RR}{\mathbb{R}}
\newcommand{\ZZ}{\mathbb{Z}}

\newcommand{\vA}{\mathcal{A}}
\newcommand{\vB}{\mathcal{B}}
\newcommand{\vC}{\mathcal{C}}

\newcommand{\vG}{\mathcal{G}}
\newcommand{\vH}{\mathcal{H}}

\newcommand{\vK}{\mathcal{K}}

\newcommand{\vS}{\mathcal{S}}

\newcommand{\Aut}{\operatorname{Aut}}

\newcommand{\Homeo}{\operatorname{Homeo}}
\newcommand{\LAut}{\operatorname{LAut}}
\newcommand{\LHomeo}{\operatorname{LHomeo}}
\newcommand{\LMod}{\operatorname{LMod}}
\newcommand{\Mod}{\operatorname{Mod}}
\newcommand{\PAut}{\operatorname{PAut}}

\newcommand{\SAut}{\operatorname{SAut}}
\newcommand{\SHomeo}{\operatorname{SHomeo}}
\newcommand{\SMod}{\operatorname{SMod}}
\newcommand{\tr}{\operatorname{tr}}

\newcommand{\lra}{\longrightarrow}

\newcommand{\sm}{\setminus}
\newcommand{\ol}[1]{\overline{#1}}
\newcommand{\wt}[1]{\widetilde{#1}}

\newcommand{\id}{\text{id}}

\newcommand{\abs}[1]{\left\lvert #1 \right\rvert}
\newcommand{\eand}{\quad \text{ and } \quad}

\definecolor{lightgrey}{gray}{.85}
\definecolor{inkscapeFuchsia}{HTML}{FF00FF}

\newcommand{\red}{\color{Red}}

\newcommand{\green}{\color{Green}}

\newcommand{\fuchsia}{\color{inkscapeFuchsia}}

\theoremstyle{definition}

\theoremstyle{plain}
\newtheorem{thm}{Theorem}[section]
\newtheorem{lem}[thm]{Lemma}
\newtheorem{cor}[thm]{Corollary}
\newtheorem{prop}[thm]{Proposition}

\newtheorem{qu}[thm]{Question}

\theoremstyle{definition}
\newtheorem{RMK}[thm]{Remark}

\begin{document}
\title[Mapping class groups of covers with boundary]{Mapping class groups of covers with boundary and braid group embeddings}
\author{Tyrone Ghaswala}
\address{Department of Mathematics, University of Manitoba, Winnipeg, R3T 2N2, Canada}
\email{ty.ghaswala@gmail.com}
\author{Alan McLeay}
\address{Mathematics Research Unit, University of Luxembourg, Esch-sur-Alzette, Luxembourg}
\email{mcleay.math@gmail.com}
\maketitle

%%%%%%%%%%%%%%%%%%%%%%%%%%%%%%%%%%%
\begin{abstract}
We consider finite-sheeted, regular, possibly branched covering spaces of compact surfaces with boundary and the associated liftable and symmetric mapping class groups.  In particular, we classify when either of these subgroups coincides with the entire mapping class group of the surface. As a  consequence, we construct infinite families of non-geometric embeddings of the braid group into mapping class groups in the sense of Wajnryb.  Indeed, our embeddings map standard braid generators to products of Dehn twists about curves forming chains of arbitrary length. As key tools, we use the Birman-Hilden theorem and the action of the mapping class group on a particular fundamental groupoid of the surface.
\end{abstract}

%%%%%%%%%%%%%%%%%%%%%%%%%%%%%%%%%%%

\section{Introduction}\label{Introduction}

The {\it mapping class group} $\Mod(\Sigma,\vB)$ of a compact orientable surface $\Sigma$ with finitely many marked points $\vB$ is the group of orientation-preserving homeomorphisms of $\Sigma$ preserving $\vB$ setwise, up to isotopies that preserve $\vB$.  If $\Sigma$ has non-empty boundary, then homeomorphisms and isotopies must fix the boundary pointwise.  We may sometimes denote a genus $g$ surface with $m$ boundary components by $\Sigma_g^m$.

Leveraging covering spaces to study mapping class groups has been a fruitful endeavour. See the work of Aramayona-Leininger-Souto \cite{ALS}, Bigelow-Budney \cite{BB}, Brendle-Margalit \cite{BM}, Brendle-Margalit-Putman \cite{BMP}, Endo \cite{Endo}, Morifuji \cite{Morifuji}, and Stukow \cite{Stukow1,Stukow2} to name only a few.

Let $p:\wt \Sigma \to \Sigma$ be a finite-sheeted, regular covering space of surfaces with deck group $D$, possibly branched at $\vB \subset \Sigma$. Recall that a homeomorphism $f:\wt\Sigma \to \wt \Sigma$ is called {\it fibre-preserving} if $p(x) = p(y)$ implies $pf(x) = pf(y)$.  Let $\SMod(\wt\Sigma) < \Mod(\wt\Sigma)$ be the subgroup consisting of isotopy classes of fibre-preserving homeomorphisms, called the {\it symmetric mapping class group}.  Let $\LMod(\Sigma,\vB) < \Mod(\Sigma,\vB)$ be the subgroup consisting of isotopy classes of homeomorphisms that lift to boundary preserving homeomorphisms of $\wt\Sigma$, called the {\it liftable mapping class group}.

Surveying the landscape of results, one notices that the fertile ground often occurs when at least one of the liftable or symmetric mapping class groups coincides with the mapping class groups $\Mod(\Sigma, \vB)$ or $\Mod(\wt \Sigma)$ respectively.  This leads to the following natural questions:

\begin{itemize}
\item When does $\LMod(\Sigma,\vB) = \Mod(\Sigma,\vB)$?
\item When does $\SMod(\wt\Sigma) = \Mod(\wt\Sigma)$?
\item If equality does not hold, when are these subgroups finite-index?
\end{itemize}

These questions are answered in Theorems \ref{ClassyLMod} and \ref{ClassySMod} of this paper in the case where the surfaces have non-empty boundary.

The first author and Winarski classified all cyclic branched covers of the sphere with the property that $\LMod(\Sigma_0,\vB) = \Mod(\Sigma_0,\vB)$ \cite[Theorem 1.1]{GW}.  Birman and Hilden proved that if $g\geq 3$, where $g$ is the genus of $\wt \Sigma$, there are no finite cyclic covers $p:\wt\Sigma \to \Sigma_0$ of the sphere such that $\SMod(\wt\Sigma) = \Mod(\wt\Sigma)$ \cite[Theorem 6]{BH}.  After completing the proof, they make the following remark:\\
\\
{\it The possibility remains that if we relax the requirements on $(p,\Sigma_0,\wt\Sigma)$ to admit coverings of other Riemann surfaces, or to admit all regular coverings, or to admit non-regular coverings that we will have better luck. (However we conjecture that all such efforts will fail).}\\
\\
As we will see, our results agree with their sentiment.

\subsection{Main Results}
From now on, let $p:\wt\Sigma \to \Sigma$ be a non-trivial, finite-sheeted, regular cover of a surface $\Sigma$.  We will take both $\wt \Sigma$ and $\Sigma$ to be compact with non-empty boundary, possibly branched at $\vB \subset \Sigma$.

\subsection*{The Burau covers} \label{burau_def}  Let ${\bf D}_n$ be a disk with $n$ points removed, and enumerate the punctures.  Pick a point $x \in \partial {\bf D}_n$ and let $\gamma_i \in \pi_1({\bf D}_n,x)$ be the homotopy class of a loop surrounding only the $i$th puncture anti-clockwise.  Then $\{\gamma_1,\ldots,\gamma_n\}$ generates $\pi_1({\bf D}_n,x)$.  For each $k\geq 2$, define a homomorphism $q_k:\pi_1({\bf D}_n,x) \to \ZZ/k\ZZ$ by $q_k(\gamma_i) = 1$ for all $i$.  The kernel of $q_k$ determines a $k$-sheeted cyclic branched cover  $p_k:\Sigma_g^m \to \Sigma_0^1$ branched at $n$ points.  Here $m = \gcd(n,k)$ and $g = 1-\frac 12(k+n+m-nk)$.  We will call such a cover a {\it $k$-sheeted Burau cover}.  The $2$-sheeted Burau covers are usually referred to as \emph{hyperelliptic covers} of a disk as in each case the non-trivial element of the deck group is a hyperelliptic involution.

We note that the Burau covers have previously been considered by McMullen to study certain unitary representations of the braid group \cite{McMullen}.

\begin{thm}\label{ClassyLMod}
Let $p:\wt \Sigma \to \Sigma$ be a non-trivial, finite-sheeted, regular covering space of compact surfaces with boundary, possibly branched at $\vB \subset \Sigma$. Then
\begin{enumerate}
\item[\normalfont (i)] $\LMod(\Sigma,\vB) = \Mod(\Sigma,\vB)$ if and only if $p$ is a Burau cover, and
\item[\normalfont (ii)] $\LMod(\Sigma, \vB)$ is always finite-index in $\Mod(\Sigma, \vB)$.
\end{enumerate}
\end{thm}

For a general covering space $p:\wt X \to X$, a homeomorphism $f:X \to X$ lifts to a homeomorphism $\tilde f:\wt X \to \wt X$ if and only if $f_*p_*\pi_1(\wt X,\tilde x) = p_*\pi_1(\wt X,\tilde x)$.  It is therefore tempting to characterise liftable mapping classes by their action on the fundamental group of the base space.  This approach ultimately falls short since it may be the case that a homeomorphism of the base space $\Sigma$ lifts to a homeomorphism of $\wt \Sigma$,  but none of its lifts fix the boundary of $\wt \Sigma$ pointwise.  Indeed, consider the 2-sheeted unbranched cover of an annulus $A$ by an annulus $\wt A$.  Any representative Dehn twist of $T \in \Mod(A)$ acts trivially on $\pi_1(A,x)$, however neither of its lifts preserve the boundary of $\wt A$ pointwise (see Figure \ref{AnnularTwist}).  Therefore $T \notin \LMod(A)$.  On the other hand, $T^2 \in \LMod(A)$.

\begin{figure}[t]
\begin{center}
\labellist \small \hair 1pt
	\pinlabel {$\wt T$} at 300 305
	\pinlabel {$T$} at 300 80
    \endlabellist
\includegraphics[scale=0.3]{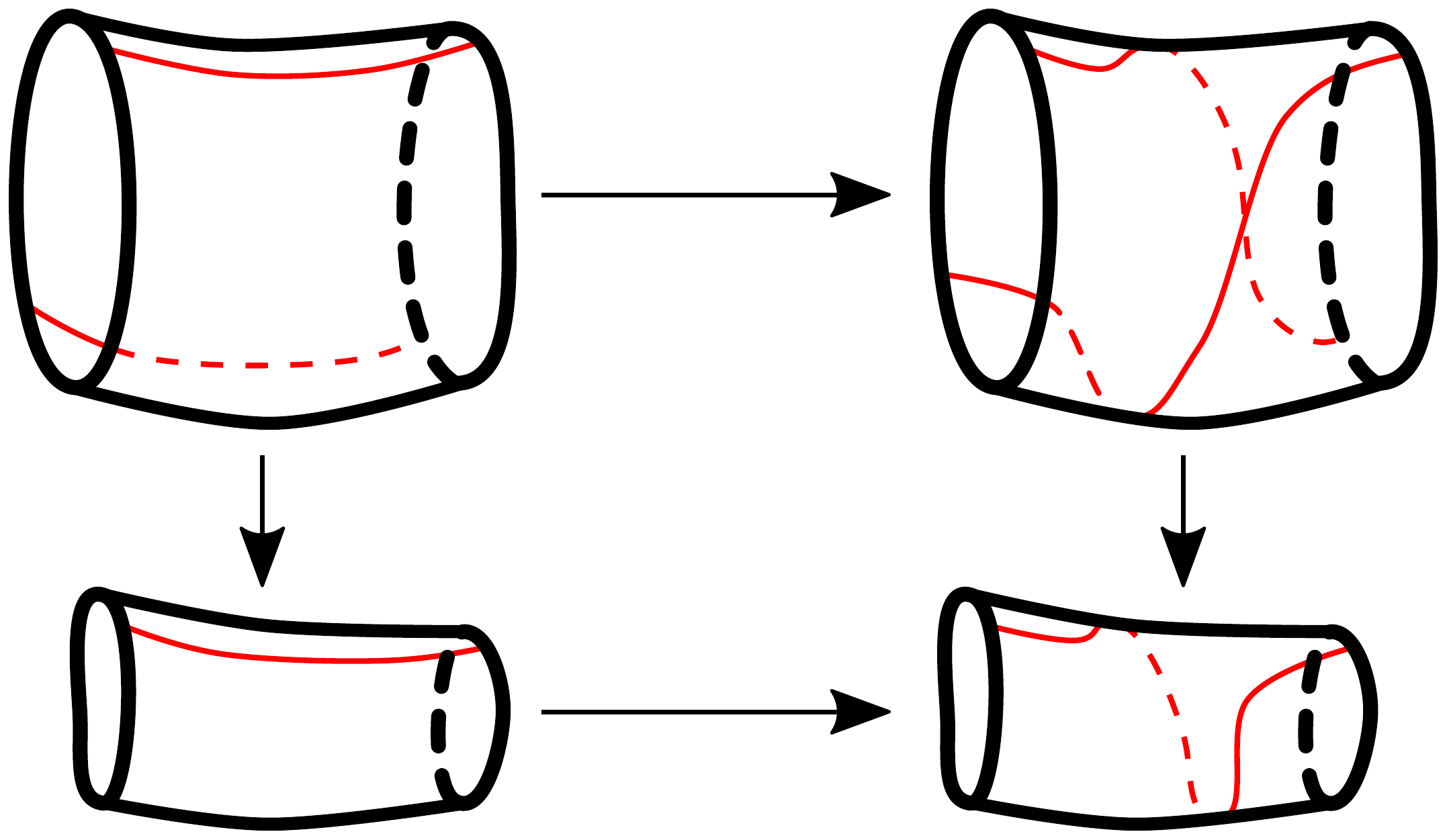}
\end{center}
\caption{One of the lifts of a Dehn twist $T \in \LHomeo^+(A)$ that does not preserve the boundary of $\wt A$ pointwise.}
\label{AnnularTwist}
\end{figure}

In order to get around this issue, we study the action of $\Mod(\Sigma,\vB)$ on a particular fundamental groupoid of $\Sigma$.  A key piece of the proof of Theorem \ref{ClassyLMod} is a characterisation of elements of $\LMod(\Sigma,\vB)$ in terms of their action on the fundamental groupoid in Theorem \ref{lmodmain}. The corresponding result for the symmetric mapping class group is as follows.

\begin{thm}\label{ClassySMod}
Let $p:\wt \Sigma \to \Sigma$ be a non-trivial, finite-sheeted, regular, possibly branched covering space of compact surfaces with boundary.  Then
\begin{enumerate}
\item[\normalfont (i)] $\SMod(\wt \Sigma) = \Mod(\wt \Sigma)$ if and only if $\wt \Sigma$ is a disk, an annulus, or $p:\Sigma_1^1 \to \Sigma_0^1$ is the hyperelliptic cover.
\item[\normalfont (ii)] Otherwise, $\SMod(\wt\Sigma)$ is infinite-index in $\Mod(\wt\Sigma)$.
\end{enumerate}
\end{thm}
As shown, the property that $\SMod(\wt\Sigma) = \Mod(\wt\Sigma)$ is rare, echoing Birman and Hilden's sentiment in the case of surfaces with boundary.

The techniques used in this paper unfortunately do not naturally extend to surfaces without boundary. What the precise analogous statements of Theorems \ref{ClassyLMod} and \ref{ClassySMod} should be remains an intriguing open question. 

\subsection{Braid group embeddings}

The braid group $B_n$ on $n$ strands is isomorphic to the mapping class group $\Mod({\bf D}_n)$, where ${\bf D}_n$ is a disk with $n$ punctures.  There is a standard embedding of the braid group in the mapping class group, which sends each standard generator to a Dehn twist.  Such an embedding is called {\it geometric}.

A question of Wajnryb asks whether there are non-geometric embeddings of the braid group in a mapping class group \cite{Wajnryb}.  Non-geometric embeddings have since been constructed in the works of B\"odigheimer-Tillman \cite{BT}, Kim-Song \cite{KS}, Song \cite{Song}, Song-Tillman \cite{ST}, and Szepietowski \cite{Szepietowski}.

Using Theorem \ref{ClassyLMod} and the Birman-Hilden theorem, we construct a family of non-geometric embeddings of the braid group.  Let $\sigma_1,\ldots,\sigma_{n-1}$ be the standard braid generators for $B_n$.

A \emph{$k$-chain} on a surface $\Sigma$ is a sequence of simple closed curves $\{a_1,\ldots,a_k\}$ on $\Sigma$ such that $i(a_i,a_j) = 1$ if $\abs{i-j} = 1$ and $i(a_i,a_j) = 0$ otherwise.  Here $i(a_i,a_j)$ is the geometric intersection number of the curves $a_i$ and $a_j$.
A {\it $k$-chain twist} is the mapping class $T_{a_1}\cdots T_{a_k}$ where $\{a_1,\ldots,a_k\}$ is a $k$-chain.  It is true that if $k\geq 2$, then a $k$-chain twist is not equal to a single Dehn twist. 

\begin{thm}\label{BETAK}
Let $n \geq 2$.  For each $k \geq 2$, there exists a surface $\Sigma$ and an injective homomorphism $\beta_k:B_n \to \Mod(\Sigma)$ such that $\beta_k(\sigma_i)$ is a $(k-1)$-chain twist for all $i$.
\end{thm}

When $k=2$, this embedding coincides with the standard geometric embedding.  The embedding when $k=3$ was independently arrived at by Kim-Song \cite{KS}.

\begin{figure}[t]
\begin{center}
\labellist \hair 1pt
	\pinlabel {$a_1$} at 270 130
	\pinlabel {$a_2$} at 410 50
	\pinlabel {$a_3$} at 540 147
	\pinlabel {$b_1$} at 35 115
	\pinlabel {$b_2$} at 410 205
	\pinlabel {$b_3$} at 700 130
    \endlabellist
\includegraphics[scale=0.35]{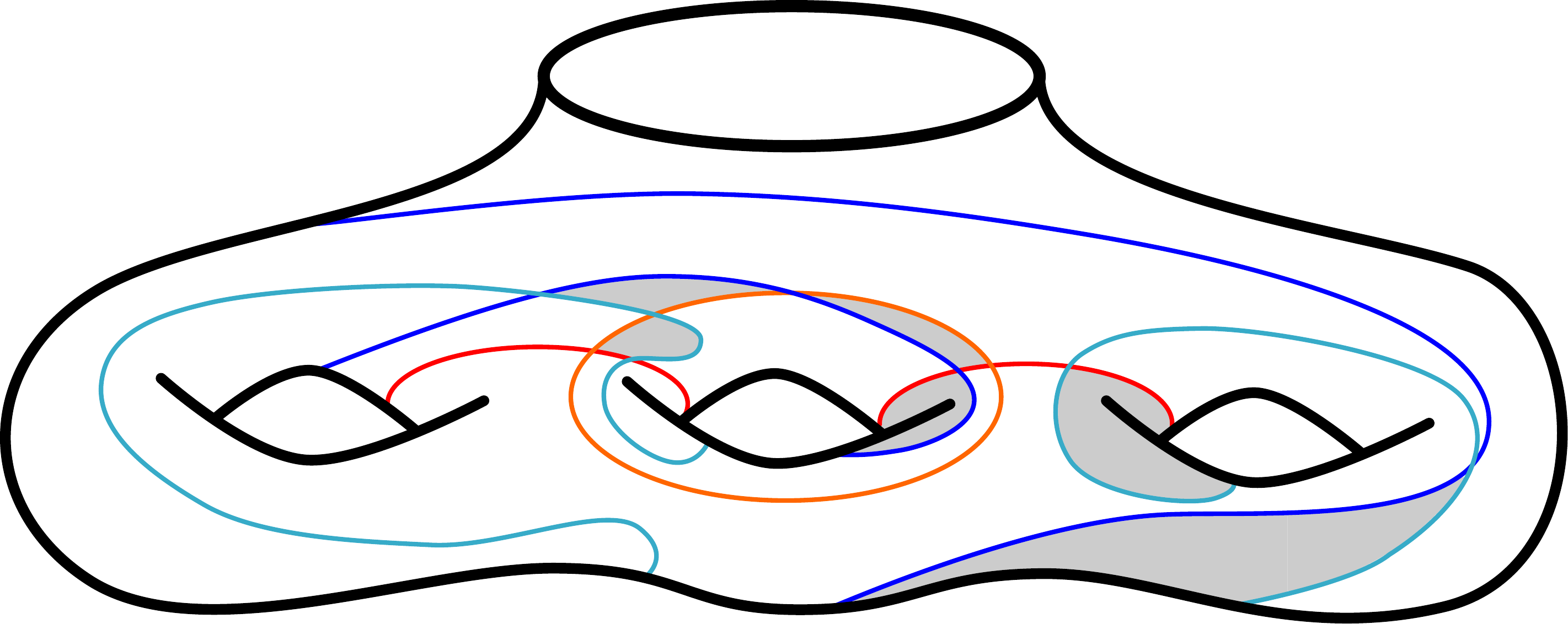}
\end{center}
\caption{The two $3$-chains $\vA = \{a_1,a_2,a_3\}$ and $\vB = \{b_1,b_2,b_3\}$ define chain twists $T_\vA := T_{a_1}T_{a_2}T_{a_3}$ and $T_\vB := T_{b_1}T_{b_2}T_{b_3}$ such that $T_\vA T_\vB T_\vA = T_\vB T_\vA T_\vB$.} 
\label{mesh_on_a_surface}
\end{figure}

Using these embeddings as inspiration, we define a combinatorial condition on two $k$-chains that imply their respective chain twists satisfy a braid relation.  The condition is defined in Section \ref{IntersectionData} and proven to be sufficient in Proposition \ref{ChainBraid}.  Figure \ref{mesh_on_a_surface} shows two 3-chains satisfying the combinatorial condition, so their chain twists satisfy a braid relation.

\subsection{Outline of the paper}
In Section \ref{LiftProject} we review some facts about lifting and projecting homeomorphisms and the Birman-Hilden theorem.  In Section \ref{section_groupoids} we review some basic facts about groupoids and prove some useful lemmas regarding fundamental groupoids and covering spaces, including a version of the Birman-Hilden theorem for automorphisms of groupoids (Lemma \ref{groupoid_birman_hilden}).  The main classification theorems are proved in Section \ref{section_classification}.  The non-geometric embeddings of the braid groups arising from the Burau covers are investigated in Section \ref{EBG}.  The paper concludes with a list of open questions relating to the braid group embeddings.

\subsection*{Acknowledgements}
We would like to thank Javier Aramayona, Paolo Bellingeri, Joan Birman, Tara Brendle, Adam Clay, David McKinnon, Doug Park, and Nick Salter for insightful conversations and support.  We also thank the anonymous referee for the numerous helpful comments and suggestions for improving the paper.  The second author would like to thank his supervisor Tara Brendle for her guidance and Doug Park's NSERC Discovery Grant for support to visit the University of Waterloo.

\section{The Birman-Hilden theorem with boundary}\label{LiftProject}
Let $\Sigma$ be an orientable compact surface, possibly with boundary $\partial \Sigma$, and finitely many marked points $\vB \subset \Sigma \sm \partial\Sigma$.  Denote by $\Homeo(\Sigma)$ and $\Homeo^+(\Sigma)$ the group of homeomorphisms of $\Sigma$ and the group of orientation-preserving homeomorphisms of $\Sigma$ respectively.  If $\vB$ appears in the argument of $\Homeo$, then we require the homeomorphisms to fix $\vB$ setwise.  If $\partial\Sigma$ appears, we require the homeomorphisms fix $\partial\Sigma$ pointwise.  Note that $\Homeo^+(\Sigma,\partial\Sigma) = \Homeo(\Sigma,\partial\Sigma)$.

Let $p:\wt \Sigma \to \Sigma$ be a regular cover with finite deck group $D \subset \Homeo^+(\wt \Sigma)$, possibly branched at $\vB \subset \Sigma \sm \partial \Sigma$.  A homeomorphism $f \in \Homeo^+(\wt \Sigma)$ is \emph{fibre-preserving} if whenever $p(x) = p(y)$, $pf(x) = pf(y)$ for all $x,y \in \wt \Sigma$.  Let $\SHomeo^+(\wt\Sigma)$ be the subgroup consisting of fibre-preserving homeomorphisms.  There is a homomorphism $\Pi: \SHomeo^+(\wt\Sigma) \to \Homeo(\Sigma,\vB)$ given by $\Pi(\tilde f)(x) = p\tilde f(\tilde x)$ for any $\tilde x$ such that $p(\tilde x) = x$.  It follows that $\ker(\Pi) = D$.  If $\Pi(\tilde f) = f$, then the square
\begin{center}
\begin{tikzcd}
    \wt \Sigma \arrow{r}{\tilde f} \arrow{d}{p} & \wt \Sigma \arrow{d}{p}\\
    \Sigma \arrow{r}{f} & \Sigma
\end{tikzcd}
\end{center}
commutes.  That is, $p\tilde f = fp$. Furthermore, it is true that $\SHomeo^+(\wt\Sigma)$ is the normaliser of $D$ in $\Homeo^+(\wt\Sigma)$.

We say a homeomorphism $f \in \Homeo^+(\Sigma,\vB)$ \emph{lifts} if there exists a homeomorphism $\tilde f \in \Homeo^+(\wt \Sigma)$ such that $p\tilde f = fp$.  Let $\LHomeo^+(\Sigma,\vB)$ be the subgroup of $\Homeo^+(\Sigma,\vB)$ consisting of homeomorphisms that lift.  Then the image of $\Pi$ is $\LHomeo^+(\Sigma,\vB)$ so $\SHomeo^+(\wt \Sigma)/D \cong \LHomeo^+(\Sigma,\vB)$.

\subsection*{Lifting and projecting with boundary}  Suppose $\wt\Sigma$ and $\Sigma$ have non-empty boundary $\partial\wt\Sigma$ and $\partial\Sigma$ respectively.  We wish to lift and project homeomorphisms that preserve the boundary pointwise, to homeomorphisms that preserve the boundary pointwise. To that end, let $\SHomeo(\wt \Sigma,\partial\wt\Sigma) = \SHomeo^+(\wt\Sigma) \cap \Homeo(\wt\Sigma,\partial\wt\Sigma)$.

\begin{prop}\label{SCentral}
$\SHomeo(\wt\Sigma,\partial\wt\Sigma) = C \cap \Homeo(\wt\Sigma,\partial\wt\Sigma)$ where $C$ is the centraliser of the deck group $D$ in $\Homeo(\wt\Sigma)$.
\end{prop}
\begin{proof}
It suffices to show that any homeomorphism $\tilde f$ that fixes the boundary and is in the normaliser of $D$ is in the centraliser of $D$.  Let $\tilde x \in \partial\wt\Sigma$ and let $d \in D$.  It is clear then that $d(\tilde x) \in \partial\wt\Sigma$.  Since $\tilde f$ fixes the boundary pointwise we have $\tilde f^{-1} d \tilde f(\tilde x) = \tilde f^{-1}d(\tilde x) = d(\tilde x)$.  Since $\tilde f$ is in the normaliser of $D$, $\tilde f^{-1}d\tilde f \in D$.  Since the deck group acts freely on $\partial\wt\Sigma$ we have $\wt f^{-1}d \wt f = d$, completing the proof.
\end{proof}

Any fibre-preserving homeomorphism of $\wt \Sigma$ that fixes $\partial\wt\Sigma$ pointwise must project to a homeomorphism of $\Sigma$ that fixes $\partial\Sigma$ pointwise.  Since the only element of $D$ that fixes $\partial \wt\Sigma$ pointwise is the identity, restricting the domain of $\Pi$ gives us an injective homomorphism $\Pi:\SHomeo(\wt\Sigma,\partial\wt\Sigma) \to \LHomeo^+(\Sigma,\vB) \cap \Homeo(\Sigma,\partial\Sigma)$.

Define $\LHomeo(\Sigma,\partial\Sigma,\vB) = \Pi(\SHomeo(\wt\Sigma,\partial\wt\Sigma))$.  That is, the set of homeomorphisms of $\Sigma$ fixing $\partial\Sigma$ pointwise that lift to a homeomorphism of $\wt\Sigma$ fixing $\partial\wt\Sigma$ pointwise. 

While it is tempting to define $\LHomeo(\Sigma,\partial\Sigma,\vB)$ as $\LHomeo^+(\Sigma,\vB) \cap \Homeo(\Sigma,\partial\Sigma)$, in general there exist boundary preserving homeomorphisms of $\Sigma$ that lift to homeomorphisms of $\wt \Sigma$ that do not fix the boundary.

Let $\wt{\mathfrak P}:\Homeo(\wt\Sigma,\partial\wt\Sigma) \to \Mod(\wt\Sigma)$ and $\mathfrak P:\Homeo(\Sigma,\partial\Sigma,\vB)\to \Mod(\Sigma,\vB)$ be the natural quotient maps.  Define the \emph{symmetric mapping class group} and the \emph{liftable mapping class group} by $\SMod(\wt\Sigma) = \wt{\mathfrak P}(\SHomeo(\wt\Sigma,\partial\wt\Sigma))$ and $\LMod(\Sigma,\vB) = \mathfrak P(\LHomeo(\Sigma,\partial\Sigma,\vB))$ 
respectively.

\subsection{The Birman-Hilden Theorem}\label{section_birman_hilden}
Suppose $\wt\Sigma$ is closed with genus at least 2.  The Birman-Hilden theorem states that projecting homeomorphisms induces an isomorphism $\SMod(\wt\Sigma)/D \cong \LMod(\Sigma,\vB)$.  Furthermore, $\SMod(\wt\Sigma)$ is the normaliser of $D$ in $\Mod(\wt\Sigma)$. The Birman-Hilden theorem was first proved for solvable covers \cite{BH}, and later proved for all finite-sheeted, regular branched covers \cite{MH}.  A version of the Birman-Hilden theorem was proved by Winarski for possibly irregular, fully-ramified covers \cite{Winarski}, and by Aramayona, Leininger, and Souto for irregular unbranched covers \cite{ALS}.  We refer the reader to a survey by Margalit-Winarski of the Birman-Hilden theorem \cite{MW}.

When $\wt \Sigma$ and $\Sigma$ have non-empty boundary, the deck group $D$ is no longer a subgroup of $\Mod(\wt\Sigma)$ since no non-trivial element of $D$ fixes the boundary components pointwise.  In this case, the Birman-Hilden theorem takes the following form.

\begin{thm}[Birman-Hilden with boundary]\label{LModSMod}
Let $p:\wt \Sigma \to \Sigma$ be a finite-sheeted, regular covering space of compact surfaces with boundary, possibly branched at $\vB \subset \Sigma\sm\partial\Sigma$.  Then projecting homeomorphisms induces an isomorphism $\SMod(\wt\Sigma)\cong \LMod(\Sigma,\vB)$.
\end{thm}

This result is well known and a proof can be found in the thesis of the second author \cite{McLeayThesis}.  Rephrased, Theorem \ref{LModSMod} states that the isomorphism $\Pi:\SHomeo(\wt\Sigma,\partial\wt\Sigma) \to \LHomeo(\Sigma,\partial\Sigma,\vB)$ induces an isomorphism $\Pi:\SMod(\wt\Sigma) \to \LMod(\Sigma,\vB)$.

\section{The Fundamental Groupoid} \label{section_groupoids}
In this section we review some facts about groupoids before studying the automorphism groups of groupoids related to covering spaces.

\subsection{Groupoids}
Here we survey the relevant results about groupoids.  See the books \cite{Higgins} and \cite{Brown} for more details.

A \emph{groupoid} is a small category where every morphism is an isomorphism.  Equivalently, a groupoid $\vG$ is a disjoint collection of sets $\{G_{ij}\}_{i,j \in I}$ together with an associative partial operation $\cdot:G_{ij} \times G_{jk} \to G_{ik}$ such that
\begin{itemize}
\itemsep0em
\item For each $i \in I$ there is an identity $e_i \in G_{ii}$ such that $e_if = f$ and $ge_i = g$ for all $f$ and $g$ such that the products $e_if$ and $ge_i$ are defined, and
\item For each $g \in G_{ij}$ there is an inverse $g^{-1} \in G_{ji}$ such that $gg^{-1} = e_i$ and $g^{-1}g = e_j$.
\end{itemize}
We will call $I$ the \emph{object set} of $\vG$.  If $\abs{I} = 1$ (equivalently if the category has one object), then $\vG$ is a group. Define the {\it source} and \emph{target} maps, $s,t:\vG \to I$ by $s(g) =i$ and $t(g) = j$ for all $g \in G_{ij}$.

A groupoid is \emph{connected} if $G_{ij} \neq \emptyset$ for all $i,j \in I$.  Notice that $G_{ii}$ is a group for all $i \in I$, and if $\vG$ is connected then $G_{ii} \cong G_{jj}$ for all $i,j \in I$.  The groups $G_{ii}$ will be called \emph{vertex groups}.  From now on we will assume $\vG$ is a connected groupoid.

Fix an $i_0 \in I$.  For each $i \in I$ choose an element $\iota_i \in G_{i_0i}$ with $\iota_{i_0} = e_{i_0}$.  Then $\vG$ is generated by the vertex group $G_{i_0i_0}$ and $\{\iota_i\}_{i \in I}$.  In fact, every element in $G_{ij}$ is uniquely written as $\iota_i^{-1}g\iota_j$ for some $g \in G_{i_0i_0}$.  We call $\{\iota_i\}_{i \in I}$ a \emph{star} based at $i_0$.

A \emph{subgroupoid} $\vH < \vG$ is a collection of subsets $\{H_{ij} \subset G_{ij}\}_{i,j \in J}$ for some non-empty $J \subset I$ such that $\vH$ is a groupoid with the operation from $\vG$.  A subgroupoid is \emph{wide} if $J = I$.  A subgroupoid $\vH < \vG$ is \emph{normal} if $f^{-1}H_{ii}f \subset H_{jj}$ for all $f \in G_{ij}$.  It follows that normal subgroupoids of connected groupoids are wide, and $h \mapsto f^{-1}hf$ is an isomorphism of groups $H_{ii} \cong H_{jj}$.

Let $\vH$ be a connected normal subgroupoid of $\vG$.  Construct the \emph{quotient groupoid} $\vG/\vH$ to be a groupoid with one object, or a group, as follows.  Put an equivalence relation $\sim$ on $\vG$ by $a \sim b$ if there exists $x,y \in \vH$ such that $a = xby$.  The equivalence classes are called the \emph{cosets} of $\vH$ in $\vG$, and these are the elements of $\vG/\vH$.  Define an operation on the cosets by $[a][b] = [axb]$ for any $x \in \vH$ with $s(x) = t(a)$ and $t(x) = s(b)$.  This is a well defined group operation on $\vG/\vH$.

Although we will not need it, the quotient groupoid can be defined for disconnected normal subgroupoids of connected groupoids.  The only difference is that there is one object for each connected component of $\vH$ (see \cite[Chapter~12]{Higgins}).

\subsection{Automorphisms of groupoids} \label{automorphisms_of_groupoids}
Let $\vG$ and $\vH$ be groupoids with object sets $I$ and $J$ respectively.  A \emph{morphism} $\phi:\vG \to \vH$ is a functor from $\vG$ to $\vH$.  Explicitly, $\phi$ is a function $\hat \phi:I \to J$ together with functions $\hat \phi_{ij}:G_{ij} \to H_{\hat \phi(i)\hat \phi(j)}$ for all $i,j \in I$ such that $\hat \phi_{ij}(a)\hat \phi_{jk}(b) = \hat \phi_{ik}(ab)$ for all $a \in G_{ij}$ and $b \in G_{jk}$.  It follows that $\hat \phi_{ii}(e_i) = e_{\hat \phi(i)}$ and $\hat \phi_{ji}(g^{-1}) = \hat \phi_{ij}(g)^{-1}$ for all $i,j \in I$ and $g \in G_{ij}$.  Given a groupoid morphism $\phi$, we will abuse notation and denote the maps $\hat \phi$ and $\hat \phi_{ij}$ by $\phi$.  Note that $s(\phi(g)) = \phi(s(g))$ and $t(\phi(g)) = \phi(t(g))$ for all $g \in \vG$.

An \emph{automorphism} of $\vG$ is a morphism $\phi:\vG \to \vG$ with a two-sided inverse.  The set of automorphisms of $\vG$ forms a group under composition, denoted by $\Aut(\vG)$.

We now restrict our attention to connected groupoids with finite object set.  Let $G$ be a group and consider the semi-direct product $G^n \rtimes \Aut(G)$.  To set notation, the group operation on $G^n \rtimes \Aut(G)$ is given by
\[
((g_1,\ldots,g_n),\psi)((h_1,\ldots,h_n),\varphi) = ((\psi(h_1)g_1,\ldots\psi(h_n)g_n),\psi\varphi)
\]
for all $g_i,h_i \in G$ and $\psi,\varphi \in \Aut(G)$.

Define the \emph{pure automorphism group} of $\vG$ by
\[
\PAut(\vG) := \{\phi \in \Aut(\vG): \phi(i) = i \text{ for all }i \in I\}.
\]
Let $\vG$ be a connected groupoid with object set $I = \{0,1,\ldots,n\}$.  Let $G = G_{00}$ be the vertex group at $0 \in I$.  Choose a star $\{\iota_i\}_{i \in I} \subset \vG$ based at $0 \in I$ and let $g_0 = e_0 \in G$.
\begin{lem} \label{Paut_isomorphism}
The map $\theta:G^n \rtimes \Aut(G) \lra \PAut(\vG)$ given by 
\[
\theta(((g_1,\ldots,g_n),\psi))(\iota_i^{-1}a\iota_j) = \iota_i^{-1}g_i^{-1}\psi(a)g_j\iota_j
\]
is an isomorphism.
\end{lem}
Lemma \ref{Paut_isomorphism} is proved in \cite[\S 3]{AM}.  Note that the isomorphism $\theta$ depends on the choice of star.

Let $\vH < \vG$ be a normal subgroupoid.  If $\phi \in \Aut(\vG)$ is such that $\phi(\vH) \subset \vH$, then $\phi$ induces an automorphism $\ol\phi \in \Aut(\vG/\vH)$ by $\ol\phi([a]) = [\phi(a)]$.  Define the subgroup $\LAut_{\vH}(\vG) < \PAut(\vG)$ by
\[
\LAut_{\vH}(\vG) = \{\phi \in \PAut(\vG) : \phi(\vH) = \vH \text{ and }\ol\phi = \id \in \Aut(\vG/\vH)\}.
\]
Our goal is to prove, with certain restrictions on $\vG$ and $\vH$, that $\LAut_{\vH}(\vG)$ is finite-index in $\PAut(\vG)$.

As above, suppose that $\vG$ is a connected groupoid with object set $I = \{0,1,\ldots,n\}$ and let $G = G_{00}$ be the vertex group at $0 \in I$.  Let $\vH$ be a connected normal subgroupoid with vertex group $H = H_{00}$, which is a normal subgroup of $G$.

Define the subgroup $K < G^n \rtimes \Aut(G)$ by
\[
K = \{((g_1,\ldots,g_n),\psi) \in G^n \rtimes \Aut(G): \psi \in \LAut_H(G) \text{ and } g_i \in H \text{ for all $i$}\}.
\]
Here, $\LAut_H(G)$ is defined by considering a group as a groupoid with one object.

\begin{lem}\label{Laut}
Choose a star $\vS = \{\iota_i\}_{i \in I} \subset \vH$ based at $0 \in I$.  Consider the isomorphism $\theta:G^n \rtimes \Aut(G) \to \PAut(\vG)$ from Lemma \ref{Paut_isomorphism} defined by $\vS$.  Then $\theta(K) = \LAut_\vH(\vG)$.
\end{lem}
\begin{proof}
Let $k = ((h_1,\ldots,h_n),\psi) \in K$, let $h_0 = e_0 \in H$ and let $\iota_i^{-1}g\iota_j$ be an arbitrary element in $\vG$.  Then $\theta(k)(\iota_i^{-1}g\iota_j) = \iota_i^{-1}h_i^{-1}\psi(g)h_j\iota_j$. Since $\psi(g) \in H$ if and only if $g \in H$, we have $\theta(k)(\iota_i^{-1}g\iota_j) \in \vH$ if and only if $g \in H$.  Therefore $\theta(k)(\vH) = \vH$.  In $\vG/\vH$,  $[\iota_i^{-1}g\iota_j] = [g]$ for any $g \in G$.  Therefore $\ol{\theta(k)}([\iota_i^{-1}g\iota_j]) = \ol{\theta(k)}([g]) = [\psi(g)]$.  Since $\psi \in \LAut_H(G)$, $[\psi(g)] = [g]$ implying $\theta(k) \in \LAut_\vH(\vG)$.

Conversely, suppose $k = ((g_1,\ldots,g_n),\psi) \in G^n \rtimes \Aut(G)$ is such that $\theta(k) \in \LAut_\vH(\vG)$.  We have $\ol{\theta(k)}([g]) = [\psi(g)] = [g]$ for all $g \in G$, so $\psi \in \LAut_H(G)$.  Let $h\iota_j$ be an arbitrary element of $H_{0j}$.  Then $\theta(k)(h\iota_j) = \psi(h)g_j\iota_j$.  For $\theta(k)(h\iota_j)$ to be in $\vH$, we must have $g_j \in H$.  Therefore $k \in K$, completing the proof.
\end{proof}

\begin{lem}\label{LAut_finite_index}
Let $\vG$ be a connected groupoid with object set $I = \{0,1,\ldots,n\}$ and $\vH$ a connected normal subgroupoid.  Let $G = G_{00}$ and $H = H_{00}$ as above.  Suppose $G$ is finitely generated and $H$ is finite-index in $G$.  Then $\LAut_{\vH}(\vG)$ is finite-index in $\PAut(\vG)$.
\end{lem}
\begin{proof}
By Lemma \ref{Laut}, it suffices to show that $K$ is finite-index in $G^n \rtimes \Aut(G)$.  It is easily checked that $((g_1,\ldots,g_n),\psi)$ and $((h_1,\ldots,h_n),\varphi)$ are in the same right coset of $K$ if and only if $[h_i] = [g_i]$ in $G/H$ for all $i$ and $\psi$ and $\varphi$ are in the same right coset of $\LAut_H(G)$ in $\Aut(G)$.  The result then follows from the fact that if $G$ is finitely generated and $H$ is finite-index in $G$, $\LAut_H(G)$ is finite-index in $\Aut(G)$.  
\end{proof}

\subsection{The fundamental groupoid}\label{fundamental_groupoid}

The groupoids of interest in this paper will be fundamental groupoids of finite-type surfaces (that is, surfaces with finitely generated fundamental group) with boundary. We will briefly state the definition and some properties without proof that will be useful later on.  For a full treatment, see \cite[Chapter~6]{Brown} or \cite[Chapter~6]{Higgins}.  

Let $X$ be a topological space and $A \subset X$ a subset.  The fundamental groupoid $\pi_1(X,A)$ is the set of homotopy classes of paths $\delta:([0,1],\{0,1\}) \to (X,A)$ relative to the endpoints.  Equipped with concatenation of paths as the partial operation, $\pi_1(X,A)$ is a groupoid with object set $A$.  The source and target maps are given by $s([\delta]) = \delta(0)$ and $t([\delta]) = \delta(1)$.  When $A = \{x\}$, we recover the fundamental group $\pi_1(X,x)$.  It is helpful to think of the fundamental groupoid as a fundamental group with multiple basepoints.

Like the fundamental group, the fundamental groupoid provides a functor from the category of pairs of topological spaces to groupoids.  In particular, if $f:X \to X$ is a homeomorphism such that $f(A) = A$, then $f_*:\pi_1(X,A) \to \pi_1(X,A)$ is an automorphism of the groupoid $\pi_1(X,A)$.  Furthermore, if $f,g:(X,A) \to (Y,B)$ are homotopic relative to $A$, then $f_*=g_*:\pi_1(X,A) \to \pi_1(Y,B)$.

We now shift our focus to covering spaces.  For a groupoid $\vG$ with object set $I$, define the sets $S(i) := \{g \in \vG:s(g) = i\}$ and $T(i) := \{g \in \vG: t(g) = i\}$.  The next lemma will be useful throughout the rest of the paper.

\begin{lem}\label{locally_bijective}
Let $p:\wt X \to X$ be a covering space, $A \subset X$ a subset. Consider the induced groupoid morphism $p_*:\pi_1(\wt X,p^{-1}(A)) \to \pi_1(X,A)$.  The restricted maps $p_*:S(x) \to S(p(x))$ and $p_*:T(x) \to T(p(x))$ are bijections for all $x \in p^{-1}(A)$.
\end{lem}

The lemma follows from the path and homotopy lifting properties for covering spaces, and the details are left to the reader. It is worth noting that Lemma \ref{locally_bijective} says that $p_*$ is a covering morphism (see \cite[Section 10.2]{Brown}).

Let $p:\wt X \to X$ be a finite-sheeted, regular covering space with deck group $D$.  Let $A = \{x_1,\ldots,x_k\}\subset X$, and $B = p^{-1}(A) \subset \wt X$. For each $i \in \{1,\ldots,k\}$, choose $\tilde x_i \in p^{-1}(x_i)$.  Let $\wt A = \{\tilde x_1,\ldots,\tilde x_k\}$.

 Define the groupoids
\[
\vG := \pi_1(X,A), \quad \vH := p_*(\pi_1(\wt X, \wt A)), \quad \vK := \pi_1(\wt X, B).
\]
Since $p$ is a regular cover, $\vH$ is a normal subgroupoid of $\vG$.  Recall the definition of $\LAut_\vH(\vG)$ from Section \ref{automorphisms_of_groupoids}.  It will be useful for us to identify when two elements of $\vG$ are in the same coset of $\vH$.

\begin{lem}\label{same_coset}
Let $g_1,g_2 \in \vG$ and let $\tilde g_1,\tilde g_2 \in \vK$ be the unique elements such that $s(\tilde g_1), s(\tilde g_2) \in \wt A$.  Then $t(\tilde g_1) = d_1(\tilde x_i)$ and $t(\tilde g_2) = d_2(\tilde x_j)$ for some $\tilde x_i,\tilde x_j \in \wt A$ and some $d_1,d_2 \in D$.  The elements $g_1$ and $g_2$ are in the same coset of $\vH$ if and only if $d_1 = d_2$.
\end{lem}
\begin{proof}
Suppose $d_1 = d_2$, and let $d = d_1$.  Choose any $z \in \pi_1(\wt X, \wt A)$ such that $s(z) = s(\tilde g_1)$ and $t(z) = s(\tilde g_2)$.  Let $y = d^{-1}(\tilde g_2^{-1}z^{-1}\tilde g_1)$.  Note $y \in \pi_1(\wt X,\wt A)$.  Then $\tilde g_1 = z\tilde g_2d(y)$ and $g_1 = p_*(z)g_2p_*(y)$ so $g_1$ and $g_2$ are in the same coset of $\vH$.

Conversely, suppose $g_1 = zg_2y$ for some $z,y \in \vH$.  Let $\tilde z$ and $\tilde y$ be such that $p_*(\tilde z) = z$, $p_*(\tilde y) = y$, and $s(\tilde z),t(\tilde z),s(\tilde y),t(\tilde y) \in \wt A$.  Then by Lemma \ref{locally_bijective}, $\tilde g_1 = \tilde z\tilde g_2 d_2(\tilde y)$.  Therefore $t(\tilde g_1) = d_2(t(\tilde y))$, completing the proof.
\end{proof}

Since $D$ acts freely on $B$, $D$ injects into $\Aut(\vK)$.  Abusing notation, we will identify $D$ with its image in $\Aut(\vK)$.  Define the group
\[
\SAut(\vK) := \{\phi \in \PAut(\vK): [\phi,d] = 1 \text{ for all }d \in D\},
\]
that is, $\SAut (\vK)$ is the intersection of $\PAut(\vK)$ with the centraliser of $D$ in $\Aut(\vK)$. Define a map $\Pi:\SAut(\vK) \to \LAut_\vH(\vG)$ by $\Pi(\phi)(g) = p_*\phi(\tilde g)$ where $\tilde g \in p_*^{-1}(g)$ is any choice of lift of $g$. The next lemma is a kind of Birman-Hilden theorem for groupoid automorphisms.  It will be of particular importance in Sections \ref{section_classification} and \ref{EBG}.

\begin{lem}\label{groupoid_birman_hilden}
The group homomorphism
\[
\Pi : \SAut(\vK)  \to \LAut_\vH(\vG) 
\]
is an isomorphism.
\end{lem}

Before embarking on the proof, we must set some notation.  Let $x \in B$ and $g \in S(p(x))$.  Denote by $[g\tilde]_{(x)} \in \vK$ the unique element such that $p_*\left([g\tilde]_{(x)}\right) = g$ and $s\left([g\tilde]_{(x)}\right) = x$.  Such a lift exists by Lemma \ref{locally_bijective}.

\begin{proof}[Proof of Lemma \ref{groupoid_birman_hilden}]
Let $\phi \in \SAut(\vK)$ and $g \in \vG$.  To see $\Pi(\phi)$ is a well defined set map, note that if $\tilde g_1,\tilde g_2 \in p_*^{-1}(g)$, then $\tilde g_1 = d\tilde g_2$ for some $d \in D$.  Then $p_* \phi (\tilde g_1) = p_* \phi d(\tilde g_2) = p_* d \phi(\tilde g_2) = p_* \phi (\tilde g_2)$.

For $g,h \in \vG$ with $s(h) = t(g)$, choose lifts $\tilde g$ and $\tilde h$ such that $s(\tilde h) = t(\tilde g)$.  Then $\tilde g\tilde h \in p_*^{-1}(gh)$ and $p_*(\phi)(\tilde g \tilde h) = p_* \phi(\tilde g) p_* \phi (\tilde h)$.  Therefore $\Pi(\phi):\vG \to \vG$ is a well defined groupoid morphism.  It is easily checked that $\Pi(\phi^{-1})$ is a two-sided inverse for $\Pi(\phi)$ so $\Pi(\phi) \in \Aut(\vG)$.  Since $\phi \in \PAut(\vK)$, we have
\[
\Pi(\phi)(s(g)) = s(p_*\phi(\tilde g)) = p_*\phi (s(\tilde g)) = p_*(s(\tilde g)) = s(g),
\]
so $\Pi(\phi) \in \PAut(\vG)$.

To show $\Pi(\phi) \in \LAut_\vH(\vG)$ it suffices to show $g$ and $p_*\phi(\tilde g)$ are in the same coset of $\vH$ in $\vG$.  Set $\tilde g = [g\tilde]_{(x_i)}$ for some $i$. Then $t(\tilde g) = d(x_j)$ for some $j$ and some $d \in D$.  Since $\phi \in \PAut(\vK)$, $s\phi(\tilde g) = x_i$ and $t(\phi(\tilde g)) = d(x_j)$.  Therefore by Lemma \ref{same_coset}, $g$ and $p_*\phi(\tilde g)$ are in the same coset of $\vH$.  We may now conclude $\Pi$ is a well defined set map.  For $\phi,\varphi \in \SAut(\vK)$, $\Pi(\phi)\Pi(\varphi)(g) = p_*\phi(\wt{p_*\varphi(\tilde g)}) = p_*\phi\varphi(\tilde g) = \Pi(\phi\varphi)(g)$ so $\Pi:\SAut(\vK) \to \LAut_\vH(\vG)$ is a well defined group homomorphism.

Define a map which we optimistically label $\Pi^{-1}:\LAut_\vH(\vG) \to \SAut(\vK)$ by $\Pi^{-1}(\psi)(k) = [\psi p_*(k)\tilde]_{(s(k)
)}$ for all $k \in \vK$.

For $k,l \in \vK$ with $s(l) = t(k)$, the product $[\psi p_*(k)\tilde]_{(s(k))}[\psi p_*(l)\tilde]_{(s(l))}$ is, by Lemma \ref{same_coset}, defined in $\vK$ since $\psi$ preserves each coset of $\vH$ in $\vG$.  Furthermore, $p_*[\psi p_*(kl)\tilde]_{(s(k))} = p_*\left([\psi p_*(k)\tilde]_{(s(k))}[\psi p_*(l)\tilde]_{(s(l))}\right)$ so $\Pi^{-1}(\psi):\vK \to \vK$ is a groupoid homomorphism by Lemma \ref{locally_bijective}.

It is easily checked that $\Pi^{-1}(\psi^{-1})$ is a two-sided inverse for $\Pi^{-1}(\psi)$.  Furthermore, by the definition of $\Pi^{-1}(\psi)$, $s(\Pi^{-1}(\psi)(k)) = s(k)$ for all $k \in \vK$ so $\Pi^{-1}(\psi) \in \PAut(\vK)$.

We have $p_*[\psi p_* d(k)\tilde]_{(s(dk))} = p_*d\left([\psi p_*(k)\tilde]_{(s(k))}\right)$ for all $ d\in D$ and \linebreak $s\left([\psi p_* d(k)\tilde]_{(s(dk))}\right) = s\left(d\left([\psi p_*(k)\tilde]_{(s(k))}\right)\right) = ds(k)$.  So by Lemma \ref{locally_bijective}, $\Pi^{-1}(\psi)d(k) = d\Pi^{-1}(\psi)(k)$ implying $\Pi^{-1}(\psi) \in \SAut(\vK)$.

It remains to show, as the notation suggests, that $\Pi^{-1}$ is a two-sided inverse for $\Pi$.  We have $\Pi\Pi^{-1}(\psi)(g) = p_*[\psi p_*(\tilde g)\tilde]_{(s(\tilde g))} = \psi(g)$ for all $\psi \in \LAut_\vH(\vG)$ and $g \in \vG$.

On the other hand, first note that $p_*[p_*\phi(k)\tilde]_{(s(k))} = p_* \phi(k)$ and \linebreak $s\left([p_*\phi(k)\tilde]_{(s(k))}\right) = s(\phi(k)) = s(k)$.  Therefore by Lemma \ref{locally_bijective}, $[p_*\phi(k)\tilde]_{(s(k))} = \phi(k)$.  We now have $\Pi^{-1}\Pi(\phi)(k) = [p_*\phi(k)\tilde]_{(s(k))} = \phi(k)$.

Since $\Pi^{-1}$ is a two-sided inverse for $\Pi$, we may finally conclude that $\Pi:\SAut(\vK) \to \LAut_{\vH}(\vG)$ is an isomorphism.
\end{proof}

\section{Proof of Classification Results}\label{section_classification}
Given a finite-sheeted covering space $p:\wt\Sigma \to \Sigma$ of a compact surface branched at $\vB \subset \Sigma$, let $\wt\Sigma^\circ = \wt \Sigma \sm p^{-1}(\vB)$ and $\Sigma^\circ = \Sigma\sm \vB$.  Note that since the set of branch points $\vB$ is always finite, the resulting surfaces $\wt \Sigma^\circ$ and $\Sigma^\circ$ are always of finite type.  Abusing notation, denote the resulting unbranched cover $p:\wt\Sigma^\circ \to \Sigma^\circ$.  It is easy to see that restricting homeomorphisms gives an isomorphism $\LMod(\Sigma,\vB) \cong \LMod(\Sigma^\circ)$.  It follows from Theorem \ref{LModSMod} that if two fibre-preserving homeomorphisms are isotopic, then they are isotopic through fibre-preserving homeomorphisms.  In particular, between two isotopic fibre-preserving homeomorphisms, there is an isotopy that preserves $p^{-1}(\vB)$ pointwise.  Therefore, restricting homeomorphisms induces an isomorphism $\SMod(\wt\Sigma) \cong \SMod(\wt\Sigma^\circ)$.  With this in mind, we will move back and forth between the branched and unbranched covers without much indication.  

As hinted at above, to characterise when a homeomorphism lifts to a homeomorphism that fixes boundary components, we must look at the action of a homeomorphism on the fundamental groupoid $\pi_1(\Sigma^\circ,A)$ for a specific choice of basepoints $A \subset \Sigma^\circ$.

\subsection{Action on the fundamental groupoids}\label{action_groupoids}
Suppose $\partial \Sigma^\circ$ has $m$ components.  Let $A = \{x_0,x_1,\ldots,x_{m-1}\} \subset \partial \Sigma^\circ$ be such that each component contains exactly one of the $x_i$.  For each $x_i$, choose a point $\tilde x_i \in p^{-1}(x_i)$ and let $\wt A = \{\tilde x_0,\tilde x_1,\ldots,\tilde x_{m-1}\} \subset \partial \wt \Sigma^\circ$.  Let $B = p^{-1}(A)$ and denote the fundamental groupoids by $\vG = \pi_1(\Sigma^\circ,A)$, $\vH = p_*\pi_1(\wt\Sigma^\circ,\wt A)$, and $\vK = \pi_1(\wt\Sigma^\circ,B)$ as in Section \ref{fundamental_groupoid}.  Homeomorphisms that are isotopic relative to the basepoints of the fundamental groupoid induce equal automorphisms of the fundamental groupoid.  Therefore there are homomorphisms $\Mod(\Sigma^\circ) \to \Aut(\vG)$ and $\Mod(\wt\Sigma^\circ) \to \Aut(\vK)$ given by the action of representative homeomorphisms on the respective groupoids.

\begin{lem}\label{commutative_groupoid_diagram}
There is a commutative diagram
\begin{center}
\begin{tikzcd}
    \SMod(\wt\Sigma) \arrow[hookrightarrow]{r}{\wt \Psi} \arrow{d}{\Pi}[swap]{\cong} & \SAut(\vK) \arrow{d}{\Pi}[swap]{\cong} \\
    \LMod(\Sigma,\vB) \arrow[hookrightarrow]{r}{\Psi} & \LAut_{\vH}(\vG)
\end{tikzcd}
\end{center}
such that 
\begin{itemize}
\item the vertical maps are the Birman-Hilden isomorphisms from Theorem \ref{LModSMod} and Lemma \ref{groupoid_birman_hilden},
\item the map $\wt \Psi$ is an injection given by composing the action of $\SMod(\wt\Sigma^\circ)$ on $\vK$ with the isomorphism $\SMod(\wt\Sigma) \cong \SMod(\wt\Sigma^\circ)$, and
\item the map $\Psi$ is an injection given by composing the action of $\LMod(\Sigma^\circ)$ on $\vG$ with the isomorphism $\LMod(\Sigma,\vB) \cong \LMod(\Sigma^\circ)$.
\end{itemize}
\end{lem}
\begin{proof}
Since representative homeomorphisms of $\SMod(\wt\Sigma^\circ)$ fix the basepoints $B\subset \partial\wt\Sigma$ pointwise, the image of $\wt\Psi$ is contained in $\PAut(\vK)$.  It follows from Proposition \ref{SCentral} and the fact that isotopic homeomorphisms relative to the basepoints induce the same groupoid automorphism, that the image of $\wt\Psi$ is contained in $\SAut(\vK)$.  Since $\vK$ has at least one basepoint on each boundary component of $\wt \Sigma$, $\wt\Psi$ is injective \cite[Theorem~3.1.1]{KK}.  It now suffices to show $\Pi\wt\Psi\Pi^{-1} = \Psi:\LMod(\Sigma,\vB) \to \LAut_\vH(\vG)$.  Let $[f] \in \LMod(\Sigma,\vB)$, $g \in \vG$, and $\tilde g \in p_*^{-1}(g) \subset \vK$.  After identifying homeomorphisms of $\Sigma$ and $\wt\Sigma$ with their restrictions to $\Sigma^\circ$ and $\wt\Sigma^\circ$ respectively we have
\[
\Pi\wt\Psi\Pi^{-1}([f])(g) = \Pi\wt\Psi([\tilde f])(g) = p_*\tilde f_*(\tilde g) = f_*p_*(\tilde g) = \Psi([f])(g).
\]
Therefore $\Pi\wt\Psi = \Psi \Pi$, completing the proof.
\end{proof}

\subsection{The case where everything lifts}

We now move on to proving Theorem \ref{ClassyLMod}. As such, we return to the setting of the original branched cover $p:\wt\Sigma \to \Sigma$ branched at $\vB \subset \Sigma\sm \partial \Sigma$.  If $[f] \in \Mod(\Sigma,\vB)$ we will abuse notation and denote by $f_* \in \Aut(\pi_1(\Sigma^\circ,A))$ the automorphism induced by any representative homeomorphism for $[f]$.  The abuse of notation is legal since any representative homeomorphism for $[f]$ fixes $A$ pointwise, and isotopic homeomorphisms induce the same groupoid automorphism.

\begin{thm}\label{lmodmain}
$\LMod(\Sigma,\vB) = \{[f] \in \Mod(\Sigma,\vB) :f_* \in \LAut_\vH(\vG)\}$.
\end{thm}
\begin{proof}
By Lemma \ref{commutative_groupoid_diagram}, it suffices to show that if $f_* \in \LAut_\vH(\vG)$, then $[f] \in \LMod(\Sigma,\vB)$.  The homeomorphism $f$ lifts since $f_*(p_*\pi_1(\wt\Sigma,\tilde x_0)) = p_*\pi_1(\wt\Sigma,\tilde x_0)$.  Let $\tilde f$ be the lift of $f$ such that $\tilde f(\tilde x_0) = \tilde x_0$.  It remains to show $\tilde f$ fixes $\partial \wt \Sigma$ pointwise.

Let $\tilde y \in \partial \wt \Sigma$.  Note that $y = p(\tilde y)$ is in the same component of $\partial \Sigma$ as $x_i$ for some $i$.  Choose $\sigma \in \pi_1(\Sigma,\partial\Sigma)$ such that $s(\sigma) =y$, $t(\sigma) = x_i$ and $\sigma$ is represented by an arc completely contained in $\partial \Sigma$ so $f_*(\sigma) = \sigma$.  Choose $\tilde \delta \in \pi_1(\wt\Sigma,\partial\wt\Sigma)$ such that $s(\tilde \delta) = \tilde x_0$ and $t(\tilde \delta) = \tilde y$.  Let $\delta = p_*(\tilde \delta)$ and note that $\delta\sigma \in \vG$.

Let $\tilde \sigma$ be the unique lift of $\sigma$ such that $s(\tilde \sigma) = \tilde y$.  We have $p_*\tilde f_*(\tilde \delta\tilde\sigma) = f_*(\delta\sigma)$, $p_*(\tilde\delta\tilde\sigma)= \delta\sigma$, $s(\tilde f_*(\tilde \delta\tilde \sigma)) = s(\tilde\delta\tilde\sigma) = \tilde x_0$, and $t(f_*(\delta\sigma)) = t(\delta\sigma)$.  Since $f_*(\delta\sigma)$ and $\delta\sigma$ are in the same coset of $\vH$, Lemma \ref{same_coset} implies $t(\tilde f_*(\tilde\delta\tilde\sigma)) = t(\tilde\delta\tilde\sigma)$, implying $t(\tilde f_*(\tilde \sigma)) = t(\tilde \sigma)$.  Since $p_*\tilde f_*(\tilde\sigma) = p_*(\tilde\sigma)$, we have $\tilde f_*(\tilde \sigma) = \tilde \sigma$ by Lemma \ref{locally_bijective}.  Therefore $\tilde f(\tilde y) = s(\tilde f_*(\tilde\sigma)) = s(\tilde \sigma) = \tilde y$, completing the proof.
\end{proof}

If $\Sigma$ has one boundary component we get the following well-known corollary.

\begin{cor} \label{one_boundary}
Suppose $\Sigma$ has one boundary component.  Choose a basepoint $x \in \partial\Sigma^\circ$ and $\tilde x \in p^{-1}(x)$.  Then
\[
\LMod(\Sigma,\vB) = \{[f] \in \Mod(\Sigma,\vB) : qf_* = q\}
\]
where $q:\pi_1(\Sigma^\circ,x) \to \pi_1(\Sigma^\circ,x)/p_*\pi_1(\wt\Sigma^\circ,\tilde x)$ is the quotient map and $f_*$ is the induced map on $\pi_1(\Sigma^\circ,x)$.
\end{cor}
\begin{proof}
The condition $qf_* = q$ is equivalent to $f_*$ acting trivially on the cosets of $p_*\pi_1(\wt\Sigma^\circ,\tilde x)$ in $\pi_1(\Sigma^\circ,x)$.  The result then follows from Theorem \ref{lmodmain}.
\end{proof}

The next proposition gives a direct way to check whether or not an element of $\Mod(\Sigma,\vB)$ is in $\LMod(\Sigma,\vB)$.  To that end, choose a point $x_0 \in \partial\Sigma^\circ$ and a lift $\tilde x_0 \in p^{-1}(x_0)$.  Choose a generating set $\{\gamma_1,\ldots,\gamma_n\}$ of $\pi_1(\Sigma^\circ,x_0)$.  Since the cover is regular, $\pi_1(\Sigma^\circ,x_0)/p_*\pi_1(\wt\Sigma^\circ,\tilde x_0)\cong D$.  Choose an isomorphism and let $q:\pi_1(\Sigma^\circ,x_0) \to D$ be the quotient map.

Suppose there are $m$ components of $\partial\Sigma^\circ$.  Enumerate the components not containing $x_0$ from 1 to $m-1$.  For each $i \in \{1,\ldots,m-1\}$, choose an arc $\sigma_i:[0,1] \to \Sigma^\circ$ such that $\sigma_i(0) = x_0$ and $\sigma_i(1)$ is in the $i$th boundary component.  Let $x_i = \sigma_i(1) \in \partial\Sigma^\circ$.

Let $A = \{x_0,x_1,\ldots,x_{m-1}\} \subset \partial\Sigma^\circ$.  Then the $\gamma_i$ and $[\sigma_j]$ are all elements of $\pi_1(\Sigma^\circ,A)$.  Given an element $[f] \in \Mod(\Sigma^\circ)$, $f_*[\sigma_j] = a_j[\sigma_j]$ for some $a_j \in \pi_1(\Sigma^\circ,x_0)$.

\begin{prop}\label{identifying_liftable_classes}
A mapping class $[f]$ is in $\LMod(\Sigma,\vB)$ if and only if $qf_*(\gamma_i) = q(\gamma_i)$ for all $i$ and $a_j \in \ker q$ for all $j$.
\end{prop}
\begin{proof}
Choose a lift $\tilde x_0 \in p^{-1}(x_0)$. For all $i$ choose lifts $\tilde \sigma_i$ of $\sigma_i$ such that $\tilde \sigma_i(0) = \tilde x_0$.  Let $\tilde x_i = \tilde \sigma_i(1)$ and let $\wt A = \{\tilde x_0,\tilde x_1,\ldots,\tilde x_{m-1}\}$.  Let $\vG = \pi_1(\Sigma^\circ,A)$ and $\vH = p_*\pi_1(\wt\Sigma^\circ,\wt A)$.  Then by Theorem \ref{lmodmain} $[f] \in \LMod(\Sigma,\vB)$ if and only if $f_* \in \LAut_\vH(\vG)$ where $\LAut_\vH(\vG)$ is as defined in Section \ref{automorphisms_of_groupoids}.

The condition $qf_*(\gamma_i) = q(\gamma_i)$ for all $i$ is equivalent to $f_*$ acting trivially on the cosets of $p_*\pi_1(\wt \Sigma^\circ,\tilde x_0)$ in $\pi_1(\Sigma^\circ,x_0)$.  The condition $a_j \in \ker q$ implies $a_j \in p_*\pi_1(\wt \Sigma^\circ,\tilde x_0)$ for all $i$.  The result follows from observing that $\{[\sigma_1],\ldots,[\sigma_{m-1}]\}$ is a star in $\vH$ and applying Lemma \ref{Laut}.
\end{proof}

We are now ready to prove the first of two classification results.

\begin{proof}[Proof of Theorem \ref{ClassyLMod}]
We first prove that $\LMod(\Sigma, \vB) = \Mod(\Sigma, \vB)$ if and only if $p : \wt \Sigma \to \Sigma$ is a Burau cover. Let $p:\wt \Sigma^\circ \to {\bf D}_n$ be the associated unbranched cover of a Burau cover. Let $x_0 \in \partial {\bf D}_n$ and let $\{c_1,\dots,c_n\}$ be a generating set for $\pi_1({\bf D}_n, x_0)$ such that $c_i$ is represented by a loop which surrounds only the $i$th puncture anti-clockwise. Let $H_i \in \Mod({\bf D}_n)$  be a half twist whose support (a disk with two punctures) intersects all representative loops of $c_i$ and $c_{i+1}$ and is disjoint from representatives of $c_j$ for all $j \ne i,i+1$. If we consider each $H_i$ as an automorphism of $\pi_1({\bf D}_n, x_0)$ we can assume that $H_i(c_i) = c_i c_{i+1} c_i^{-1}$ and $H_i(c_{i+1})=c_i$. From the definition of Burau covers we have that
\[
q H_i(c_i)=q(c_i c_{i+1} c_i^{-1}) = q(c_i) q( c_{i+1}) q( c_i^{-1}) = 1,
\]\[
q H_i(c_{i+1})=q(c_i) = 1, \mbox{ and}
\]\[
q H_i(c_j)=q(c_j)=1 \mbox{ for all } j \ne i,i+1.
\]
By Corollary \ref{one_boundary} we have that $H_i \in \LMod({\bf D}_n)$. It follows from the fact that the set $\{ H_1,\dots,H_{n-1} \}$ generates $\Mod({\bf D}_n)$ that every mapping class lifts, that is, $\LMod ({\bf D}_n) = \Mod ({\bf D}_n)$. As discussed at the beginning of Section \ref{section_classification} this is equivalent to showing that $\LMod(\Sigma_0^1, \vB) = \Mod(\Sigma^1_0, \vB)$.

\begin{figure}[t]
\begin{center}
\labellist\hair 1pt
	\pinlabel {$x_0$} at 158 15 
	\pinlabel {$x_1$} at 260 127 
	\pinlabel {$x_2$} at 302 99 
	\pinlabel {$x_{m-1}$} at 266 38 
	\pinlabel {\red $a_1$} at 93 16 
	\pinlabel {\color{Blue} $b_1$} at 46 26 
	\pinlabel {\red $a_2$} at 21 33 
	\pinlabel {\color{Blue} $b_2$} at 25 55 
	\pinlabel {\red $a_g$} at 28 99 
	\pinlabel {\color{Blue} $b_g$} at 77 123 
	\pinlabel {\green $c_1$} at 118 116 
	\pinlabel {\green $c_2$} at 143 116 
	\pinlabel {\green $c_n$} at 185 116 
	\pinlabel {\color{teal} $\iota_1$} at 237 115 
	\pinlabel {\color{teal} $\iota_2$} at 268 80 
	\pinlabel {\color{teal} $\iota_{m-1}$} at 228 36 
	\pinlabel {\fuchsia $d_1$} at 263 117
	\pinlabel {\fuchsia $d_2$} at 293 71
	\pinlabel {\fuchsia $d_{m-1}$} at 276 58
    \endlabellist
\includegraphics[scale=1.15]{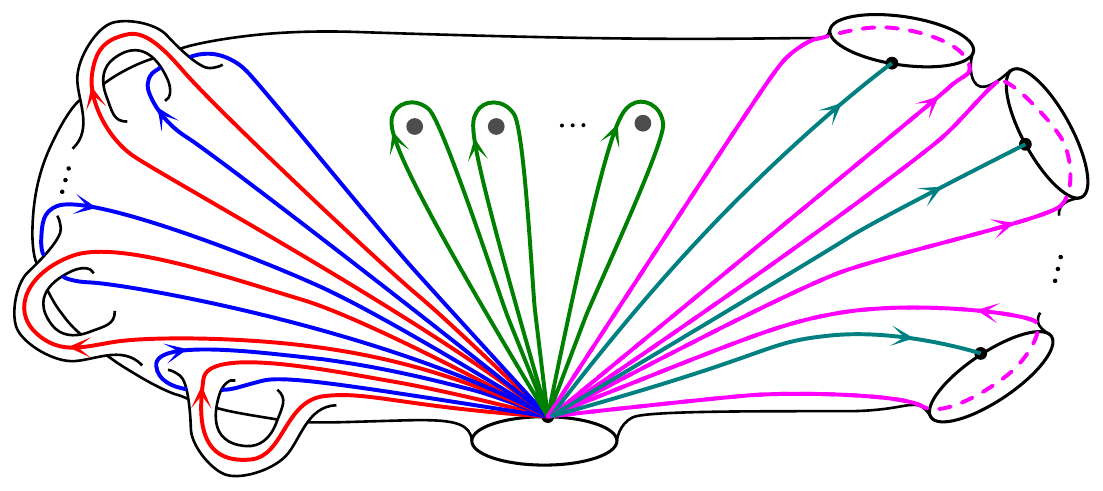}
\end{center}
\caption{A generating set for $\pi_1(\Sigma^\circ,A)$.}
\label{hindenburg}
\end{figure}
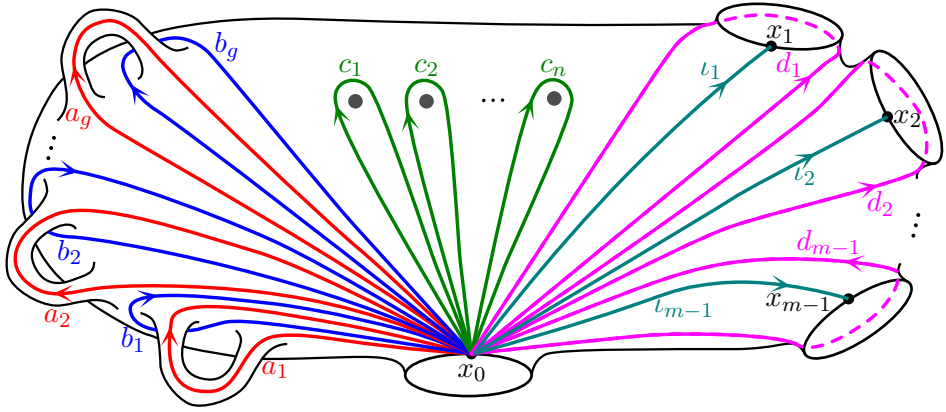

Conversely, assume $\LMod(\Sigma,\vB) = \Mod(\Sigma,\vB)$.  Suppose $\abs{\vB} = n \geq 0$ and $\Sigma$ has genus $g\geq 0$ and $m \geq 1$ boundary components. Consider the generating set 
\[
\{a_1,b_1,\ldots,a_g,b_g,c_1,\ldots,c_n,d_1,\ldots,d_{m-1},\iota_1,\ldots,\iota_{m-1}\}
\]
of the fundamental groupoid $\vG = \pi_1(\Sigma^\circ,A)$ defined in Figure \ref{hindenburg}, where $\vG$ is as in the beginning of Section \ref{action_groupoids}.  Note that $\{\iota_1,\ldots,\iota_{m-1}\}$ forms a star in $\vG$, and each $\iota_i$ is represented by a non-separating simple arc.  Denote by $\Upsilon$ the generating set $\{a_1,b_1,\ldots,a_g,b_g,c_1,\ldots,c_n,d_1,\ldots,d_{m-1}\}$ of $\pi_1(\Sigma^\circ,x_0)$.

We aim to show $\Sigma$ is a disk, and we begin by showing $m = 1$.  Suppose not, and note that every element in $\vG$ of the form $c\iota_1$ where $c \in \Upsilon$ can be represented by a simple arc with initial and terminal endpoint agreeing with $\iota_1$.  Therefore there is an element $[f] \in \Mod(\Sigma^\circ)$ such that $f_*(\iota_1) = c\iota_1$.  By Proposition \ref{identifying_liftable_classes} we have $c \in p_*\pi_1(\wt\Sigma^\circ,\tilde x_0)$ for all $c \in \Upsilon$, contradicting the assumption that $p$ is non-trivial.  Therefore $\Sigma$ must have a single boundary component.

Assume now that $m = 1$ and $\Sigma$ has positive genus $g$.  For all $i \in \{ 1,\dots,n \}$ we can find an element $[f] \in \Mod (\Sigma^\circ)$ such that $f_*(c_i)=c_1$. It follows from Corollary \ref{one_boundary} that $q(c_i)=qf_*(c_i)=q(c_1)$ for all $i$.

Since the $a_i$ and $b_i$ are represented by a simple non-separating loops based at $x_0 \subset \partial \Sigma^\circ$, there is an element of $\Mod(\Sigma^\circ)$ that sends $a_1$ to $a_i$ (or $b_i$) for any $i$.  Therefore $q(a_i) = q(b_j)$ for all $i$ and $j$.  Similarly, since $a_1b_1$ and $a_1c_1$ are also represented by simple non-separating loops based at $x_0$, there exist $[f],[h] \in \Mod(\Sigma^\circ)$ such that $f_*(a_1)=a_1b_1$ and $h_*(a_1)= a_1c_1$.  We now have that
\[
q(a_1) =qf_*(a_1) = q(a_1b_1) = q(a_1)q(b_1), 
\]
and similarly $q(a_1)= q(a_1)q(c_1)$. Therefore $q(b_1)$ and $q(c_1)$ are the identity in the deck group $D$.  We may conclude that $q(a_i)$ and $q(b_i)$ are the identity in $D$ for all $i$, and $q(c_j)$ is the identity in $D$ for all $j$.  However, this implies $\ker(q) = \pi_1(\Sigma^\circ,x_0)$, contradicting the assumption that $p$ is non-trivial.

The genus of $\Sigma$ must therefore be zero and, as shown above, $\Sigma$ has a single boundary component, that is, $\Sigma$ is a disk. We have already shown that $q(c_i)=q(c_1)$ for all $i=1,\dots,n$ and so it follows that $p:\wt \Sigma \to \Sigma$ is a Burau cover.  This completes the proof of (i).

For (ii), recall the definitions of the fundamental groupoids $\vH$ and $\vG$ from the beginning of Section \ref{action_groupoids}. Let $\Psi:\Mod(\Sigma,\vB) \to \PAut(\vG)$ be the injective homomorphism given by the action of $\Mod(\Sigma,\vB)$ on the fundamental groupoid $\vG$. By Theorem \ref{lmodmain} it follows that $\Psi(\LMod(\Sigma, \vB)) = \LAut_\vH (\vG) \cap \Psi(\Mod(\Sigma,\vB))$. We have
\begin{align*}
[\Mod(\Sigma, \vB): \LMod(\Sigma, \vB)] &= [\Psi(\Mod(\Sigma, \vB)): \Psi(\LMod(\Sigma, \vB))]\\
&\leq [\PAut(\vG):\LAut_\vH(\vG)]\\
& < \infty.
\end{align*}
The first inequality follows from the fact that for subgroups $H,K$ of a group $G$, $[G:H] \geq [K:H \cap K]$.  The second inequality is by Lemma \ref{LAut_finite_index}.
\end{proof}

While we have shown that there are infinitely many covering spaces with the property that $\LMod(\Sigma, \vB) = \Mod(\Sigma, \vB)$, it is clear that this occurs only in a distinct minority of cases. We will see in the next section that the conditions for a covering space to satisfy $\SMod(\wt \Sigma) =  \Mod(\wt \Sigma)$ are even more severe.

\subsection{The case where everything is symmetric}

In this section we prove Theorem \ref{ClassySMod}. In particular, we show that the symmetric mapping class group $\SMod(\wt \Sigma)$ coincides with the mapping class group $\Mod(\wt \Sigma)$ in a very small number of cases. To prove the result we make use of the mapping class group action on homology. Recall that the Lefschetz fixed point theorem for smooth manifolds states
\[
\sum_{p \in \operatorname{fix}(f)} i(f,p) = \sum_{i=0}^\infty (-1)^i\tr(f_*:H_i(\Sigma;\QQ) \to H_i(\Sigma;\QQ)).
\]
where $i(f,p)$ is the index of the fixed point $p$ of the homeomorphism $f$ \cite{Lefschetz}. We will apply this result to our context of surfaces with boundary.

\begin{lem} \label{number_fixed_points}
Let $\Sigma$ be a compact oriented surface with boundary. Let $f$ be a finite-order, orientation-preserving homeomorphism of $\Sigma$. Then the fixed points of $f$ are isolated and the number of fixed points is equal to
\[
1 - \tr(f_*:H_1(\Sigma;\ZZ) \to H_1(\Sigma;\ZZ)).
\]
\end{lem}

\begin{proof}
We will first prove that the fixed points are isolated. Let $f$ have order $k$ and let $\mu$ be a Riemannian metric on $\Sigma$. Define the Riemannian metric
\[
\ol \mu := \sum_{i=1}^k (f^k)^*\mu.
\]
Then $f^*\ol \mu = \ol \mu$ and so $f$ is an isometry. Since $f$ is orientation-preserving its fixed points must be isolated. Let $p \in \Sigma$ be such a fixed point and let $T_p \Sigma \cong \RR^2$ be the tangent space. Now, all orientation-preserving isometries of $\RR^2$ that fix the origin are rotations about the origin. We therefore have that $f$ induces a rotation $T_p \Sigma \to T_p \Sigma$ and so $i(f,p) = 1$.

For a surface with boundary, $H_i(\Sigma;\QQ) \cong \{0\}$ for all $i \geq 2$. Furthermore, $H_0(\Sigma;\QQ) \cong \QQ$ and $f_*:H_0(\Sigma;\QQ) \to H_0(\Sigma;\QQ)$ is the identity map. It follows that $\tr(f_*:H_0(\Sigma;\QQ) \to H_0(\Sigma;\QQ)) = 1$. Note that since the first homology group is free abelian we may replace the coefficients with $\ZZ$. Finally, since the index of each fixed point is $1$, we have that the number of fixed points is equal to
\[
1 - \tr(f_*:H_1(\Sigma;\ZZ) \to H_1(\Sigma;\ZZ))
\]
completing the proof.
\end{proof}

\begin{cor} \label{nontrivial_action}
Suppose $\Sigma$ is a compact orientable surface of genus $g$ with $m \geq 1$ boundary components other than a disk or an annulus.  Let $f \in \Homeo^+(\Sigma)$ be a finite-order, orientation-preserving homeomorphism of $\Sigma$.  Then $f$ acts non-trivially on $H_1(\Sigma;\ZZ)$.
\end{cor}
\begin{proof}
If $f$ acts trivially on $H_1(\Sigma;\ZZ)$ then $\tr(f_*:H_1(\Sigma;\ZZ) \to H_1(\Sigma;\ZZ)) = 2g + m - 1$ by choice of a natural basis of $H_1(\Sigma ; \ZZ)$. By Lemma \ref{number_fixed_points} we must have $1 - (2g+m-1) \geq 0$ and so $2g + m \leq 2$.  This only occurs when $g = 0$ and $m = 1,2$, or equivalently, when $\Sigma$ is a disk or an annulus.
\end{proof}

The next result shows that a hyperelliptic involution has a unique action on homology up to conjugation. The proof follows from Lemma \ref{number_fixed_points} and an argument similar to that of \cite[Proposition 7.15]{Primer}.

\begin{lem} \label{conjugate_hyperelliptic_with_boundary}
Let $\Sigma$ be a surface of genus $g \geq 1$ with a single boundary component. Suppose $f_1,f_2 \in \Homeo^+(\Sigma)$ are order 2 homeomorphisms such that $(f_1)_* = (f_2)_* = -I:H_1(\Sigma;\ZZ) \to H_1(\Sigma;\ZZ)$.  Then $f_1$ and $f_2$ are conjugate in $\Homeo^+(\Sigma)$.
\end{lem}

Throughout the proof of Theorem \ref{ClassySMod} we will repeatedly use the fact that if $c$ is an isotopy class of simple closed curves then every power of the Dehn twist $T_c$ is an element of $\SMod(\wt \Sigma)$ if and only if $d(c)=c$ for all $d \in D$, where $D$ is the deck group.

\begin{proof}[Proof of Theorem \ref{ClassySMod}]
We start by proving that $\SMod(\wt \Sigma) = \Mod(\wt \Sigma)$ in the three cases stated in the theorem. First, if $\wt \Sigma$ is a disk then $\SMod(\wt\Sigma) = \Mod(\wt\Sigma)$ trivially. If $\wt \Sigma$ is an annulus, let $c$ be the unique unoriented isotopy class of an essential simple closed curve. For every homeomorphism $f$ of $\wt \Sigma$ we have that $f(c) = c$. Since $\Mod(\wt\Sigma) = \langle T_c \rangle$ it follows that $\SMod(\wt\Sigma) = \Mod(\wt\Sigma)$. Finally, suppose $\wt \Sigma$ is a torus with a single boundary component, and let $\iota \in \Homeo^+(\wt\Sigma)$ be a hyperelliptic involution. There exist two simple closed curves $a$ and $b$ whose isotopy classes are fixed by $\iota$ such that the Dehn twists $T_a$, $T_b$ generate $\Mod(\wt\Sigma)$. Therefore we have that $\SMod(\wt\Sigma) = \Mod(\wt\Sigma)$.

Conversely, suppose $\SMod(\wt\Sigma) = \Mod(\wt\Sigma)$ and $\wt\Sigma$ is neither a disk nor an annulus. Suppose $\wt \Sigma$ is a surface of genus $g \ge 0$ with $m \geq 1$ boundary components. There is a generating set $\mho = \{a_1,\ldots,a_{2g},x_1,\ldots,x_{m-1}\}$ of $H_1(\wt\Sigma;\ZZ)$ where each generator $a_i$ is represented by an essential simple closed curve and each $x_i$ is the homology class of a curve isotopic to a boundary component.

Let $d$ be a non-trivial element of the deck group $D$. It must be that $d$ preserves the unoriented isotopy class of every essential simple closed curve and so we see that $d_*:H_1(\wt\Sigma;\ZZ) \to H_1(\wt\Sigma;\ZZ)$ is given by the diagonal matrix \[
\begin{bmatrix} \epsilon_1 & & \\ & \ddots & \\ & & \epsilon_{2g+m-1}\end{bmatrix}
\]
with respect to the generating set $\mho$, where $\epsilon_i = \pm 1$ for all $i$. However, since $d$ is orientation-preserving, it must preserve the orientation of every boundary component, therefore $\epsilon_i = 1$ for all $i > 2g$.  It follows from Corollary \ref{nontrivial_action} that $g \ge 1$ and there is at least one $i \in \{1,\ldots,2g\}$ such that $\epsilon_i = -1$.

We now argue that $\wt \Sigma$ must have exactly one boundary component.
If $m \ge 2$ then there is at least one element in $\mho$ that is the homology class of a boundary component. Consider the homology classes $x_1 + a_i$ and $x_1 - a_i$, where $\epsilon_i = -1$. One of these is the homology class of an essential simple closed curve $c$. Since $d_*(x_1 \pm a_i) = x_1 \mp a_i$ we have that $d_*(c) \neq \pm c$ and so the unoriented isotopy class of $c$ is not preserved by $d$. Therefore $T_c \not\in\SMod(\wt\Sigma)$ and so $\SMod(\wt \Sigma) \ne \Mod(\wt \Sigma)$, a contradiction. It follows then that $m=1$.

The next step is to show $d$ is a hyperelliptic involution.  Suppose not, then by Corollary \ref{nontrivial_action} there are $i,j \in \{1,\ldots,2g\}$ such that $\epsilon_i = 1$ and $\epsilon_j = -1$. Similar to the above argument, we can find a curve $c$ such that $T_c \notin \SMod(\wt\Sigma)$. This implies that
\[
d_* = -I:H_1(\wt\Sigma;\ZZ) \to H_1(\wt\Sigma;\ZZ)
\]
for all non-trivial $d \in D$. Suppose $d_1,d_2 \in D$ are non-trivial elements.  Then $(d_1d_2)_* = I:H_1(\wt\Sigma;\ZZ) \to H_1(\wt\Sigma;\ZZ)$, so by Corollary \ref{nontrivial_action}, $d_1 = d_2$ and $\abs{D} = 2$.

Let $\iota$ be a hyperelliptic involution of $\wt \Sigma$, then $\iota_* = -I:H_1(\wt\Sigma;\ZZ) \to H_1(\wt\Sigma;\ZZ)$.  Since any conjugate of $\iota$ is a hyperelliptic involution, it follows from Lemma \ref{conjugate_hyperelliptic_with_boundary} that $D$ is generated by a hyperelliptic involution. Finally, if $g \geq 2$ then we can find a curve that is not fixed by $\iota$ (see Figure \ref{HyperCurve}), completing the proof of the first statement.
\begin{figure}[t]
\centering
\includegraphics[scale=0.22]{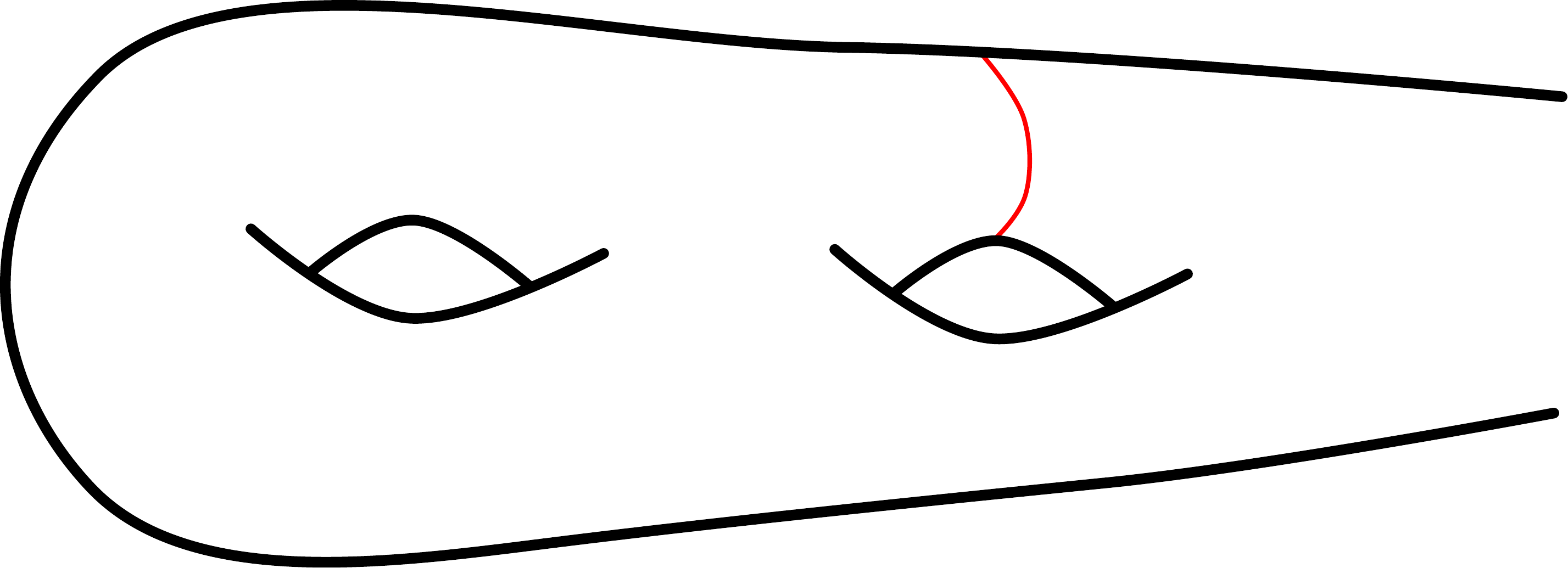}
\caption{A curve whose isotopy class is not preserved by the hyperelliptic involution for a surface with boundary $\Sigma$ of genus at least two.}
\label{HyperCurve}
\end{figure}

If $\SMod(\wt \Sigma) \ne \Mod(\wt \Sigma)$ then we may choose a Dehn twist $T_c \not \in \SMod(\wt \Sigma)$ and note that each power belongs to a different coset of $\SMod(\wt \Sigma)$ in $\Mod(\wt \Sigma)$. Since Dehn twists have infinite order, it follows that $\SMod(\wt \Sigma)$ is infinite-index in $\Mod(\wt \Sigma)$.
\end{proof}

\begin{RMK}
Suppose $p:\wt\Sigma \to \Sigma$ is a finite-sheeted, regular, possibly branched cover of compact surfaces without boundary with deck group $D$.  Combining the proof of Theorem 4 in \cite{BH} with the Neilsen realisation theorem for finite groups \cite{Kerckhoff} allows one to conclude that $\SMod(\wt\Sigma)$ is the normaliser of $D$ in $\Mod(\wt\Sigma)$.

When the surfaces in question have boundary, then $D$ is not a subgroup of $\Mod(\wt\Sigma)$.  However, $D$ and $\Mod(\Sigma)$ are both subgroups of $\Aut(\vK)$, where $\vK$ is the groupoid defined in Section \ref{action_groupoids}.

In light of both the normaliser result just stated for closed surfaces, and Theorem \ref{lmodmain} for $\LMod(\Sigma,\vB)$, we conjecture that $\SMod(\wt\Sigma) = \{[f] \in \Mod(\wt\Sigma):f_* \in \SAut(\vK)\}$.  Unfortunately, a proof seems out of reach at the moment.
\end{RMK}

\section{Non-geometric embeddings of braid groups}\label{EBG}

In this section we will investigate a family of injective homomorphisms from the braid group to mapping class groups. We will refer to such a homomorphism as a \emph{braid group embedding}. We first recall the definition of the Burau covers from Section \ref{Introduction}.

\subsection*{Burau covers}
Pick a point $x \in \partial {\bf D}_n$ and let $\gamma_i \in \pi_1({\bf D}_n,x)$ be the homotopy class of a loop surrounding solely the $i$th puncture anti-clockwise.  Then $\{\gamma_1,\ldots,\gamma_n\}$ generates $\pi_1({\bf D}_n,x)$.  For each $k\geq 2$, define a homomorphism 
\begin{align*}
q_k:\pi_1({\bf D}_n,x) & \to \ZZ/k\ZZ \\ 
\gamma_i & \mapsto 1
\end{align*} for all $i$. The kernel of $q_k$ determines a $k$-sheeted cyclic branched cover $p_k:\Sigma_g^m \to \Sigma_0^1$ branched at $n$ points. Here $m = \gcd(n,k)$ and $g = 1 - \frac 12(k + n -nk + m)$.

In Theorem \ref{ClassyLMod} it was shown that $\LMod(\Sigma, \vB) = \Mod(\Sigma, \vB)$ if and only if $\Sigma$ is a disk and $p_k : \Sigma^m_g \to \Sigma^1_0$ is a $k$-sheeted Burau cover. We can therefore define the following braid group embedding;
\[
\beta_k : B_n \cong \Mod(\Sigma_0^1, \vB) = \LMod(\Sigma_0^1, \vB) \cong \SMod(\Sigma^m_g) \hookrightarrow \Mod(\Sigma^m_g).
\]
The first isomorphism is well known, the equality comes from Theorem \ref{ClassyLMod}, and the second isomorphism is a consequence of the Birman-Hilden theorem.

Let $\{\sigma_1, \dots, \sigma_{n-1}\}$ be the standard generators of $B_n$. It is known that $\beta_2(\sigma_i) = T_{c_i}$ where $c_i$ is some non-separating curve for all $i \in \{1, \dots, n-1\}$. Furthermore, the deck group $D \cong \ZZ / 2 \ZZ$ is generated by a hyperelliptic involution \cite[Section 9.4]{Primer}.

In this section we will describe the image of the standard braid generators under $\beta_k$ where $k \ge 3$. In particular we show that $\beta_k$ is a non-geometric embedding of the braid group, that is, $\beta_k(\sigma_i)$ is \emph{not} a Dehn twist. In order to describe the image of a single braid generator it suffices to consider the embeddings
\[
\beta_{2g+1} : B_2 \hookrightarrow \Mod(\Sigma^1_g) \eand \beta_{2g+2} : B_2 \hookrightarrow \Mod(\Sigma^2_g),
\]
for integers $g > 0$. In other words, we will study the Burau covers
\[
p_{2g+1} : \Sigma^1_g \to \Sigma^1_0 \eand p_{2g+2} : \Sigma^2_g \to \Sigma^1_0,
\]
in each case branched at two points. We will deal with the two cases separately although the techniques used in each case are similar.

\subsection{Odd Burau}\label{Odd}
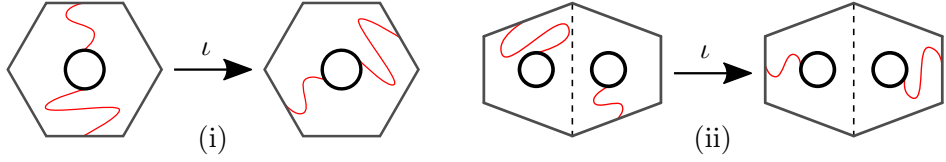
\begin{figure}[t]
\begin{center}
\labellist \hair 1pt
	\pinlabel {$\iota$} at 380 170
	\pinlabel {$\iota$} at 1320 165
	\pinlabel {(i)} at 393 0
	\pinlabel {(ii)} at 1335 0
    \endlabellist
\includegraphics[scale=0.205]{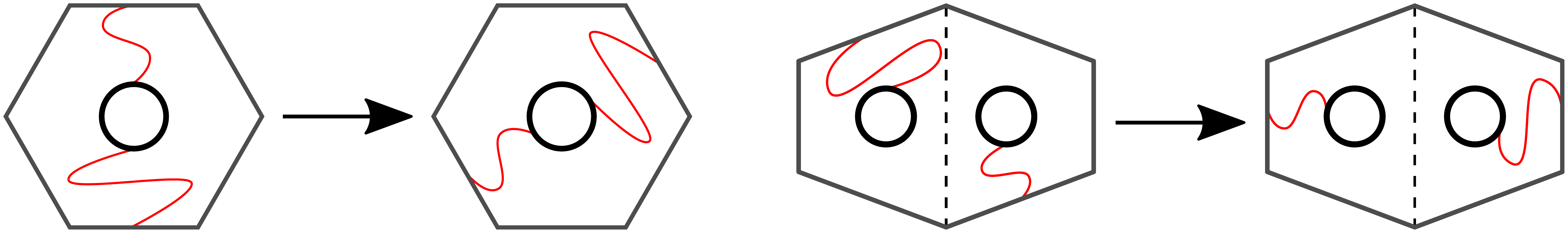}
\end{center}
\caption{(i) A generator of the deck group $D \cong \ZZ/3\ZZ$ of the $3$-sheeted Burau cover.  (ii) A generator of the deck group $D \cong \ZZ/4\ZZ$ of the 4-sheeted Burau cover.}
\label{Deck}
\end{figure}
First we consider the braid group embedding $\beta_{2g+1}$ given above. We will define an element $N$ of $\Mod(\Sigma_g^1)$ and then prove that the isomorphism
\[
\Pi : \SMod(\Sigma^1_g) \to \LMod(\Sigma^1_0, \vB)
\]
sends $N$ to the standard generator of $\Mod(\Sigma_0^1,\vB)$. Recall that we can represent a closed surface of genus $g$ by a regular $(4g+2)$-gon, centred at the origin, with opposite sides identified. If we remove an open disk about the centre we arrive at a representation of $\Sigma_g^1$. Label this representation $P$ and let $\iota$ be the anti-clockwise rotation of $P$ about its centre by $2\pi / (2g+1)$ (see Figure \ref{Deck}(i) for $\iota$ when $g= 1$). The two unique vertices of $P$ are fixed by $\iota$. We see that the quotient space $\Sigma_g^1 / \langle \iota \rangle$ is homeomorphic to $\Sigma_0^1$, and the quotient map is a covering map branched at two points. Furthermore, around both fixed points $\iota$ is locally a rotation by $2\pi / (2g+1)$ anti-clockwise, therefore the associated covering space is the $(2g+1)$-sheeted Burau cover of $\Sigma_0^1$ with deck group $D \cong \mathbb{Z} / (2g+1)\mathbb{Z}$.

We will write $p_{2g+1} : {\Sigma_g^1}^\circ \to {\bf D}_2$ for the associated unbranched cover. Let $x \in \partial {\bf D}_2$ and let $a$ and $b$ be elements of $\pi_1({\bf D}_2,x)$ such that $a$ is represented by a loop that surrounds a single marked point and $b$ is represented by a loop isotopic to $\partial {\bf D}_2$ as in Figure \ref{Generators}(i).

The elements $a$ and $b$ generate $\pi_1({\bf D}_2,x)$. Denote the full preimage $p_{2g+1}^{-1}(x)$ by $\{\wt x_i\}$ indexed by elements of $ \mathbb{Z} / (2g+1)\mathbb{Z}$ such that $\iota(\wt x_i) = \wt x_{i+1}$. Similarly we define $(p_{2g+1})_*^{-1}(a) = \{a_i\}$ and $(p_{2g+1})_*^{-1}(b) = \{b_i\}$ where $s(a_i)=s(b_i)=\tilde x_i$.  Furthermore, we have that $\iota_* (a_i) = a_{i+1}$ and $\iota_* (b_i) = b_{i+1}$, see Figure \ref{Generators}(ii). The set $\{a_i,b_i\}$, indexed by elements of $\ZZ / (2g+1) \ZZ$, generates the fundamental groupoid $\pi_1({\Sigma_g^1}^\circ, \{\wt x_i\})$, a fact which follows from Lemma \ref{locally_bijective}.

\subsection*{The odd notch}
We let $N$ denote the mapping class in $\Mod(\Sigma_g^1)$ represented by the homeomorphism that rotates the edges of $P$ by $2 \pi / (4g+2)$ and fixes the single boundary component at the centre.  See Figure \ref{Notch} for an image of $N$ when $g = 1$.

In the following lemma we will write $H$ for the half twist in $\Mod(\Sigma_0^1, \vB)$ and for the induced automorphism of $\pi_1({\bf D}_2,x)$ such that $H(a)=ba^{-1}$.
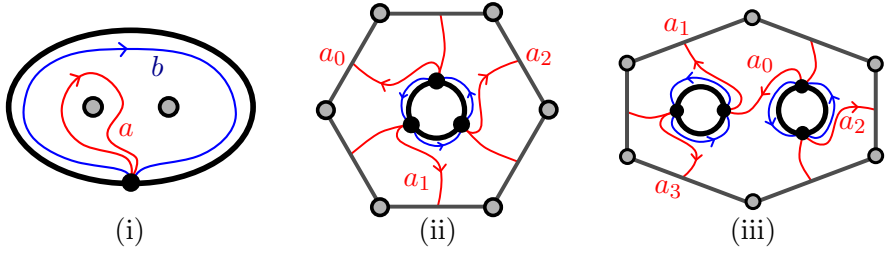
\begin{figure}[t]
\begin{center}
\labellist  \hair 1pt
    \pinlabel {(i)} at 340 -35
    \pinlabel {(ii)} at 1175 -35
    \pinlabel {(iii)} at 2030 -35
	\pinlabel {\red $a$} at 330 240
	\pinlabel {\color{DarkBlue} $b$} at 415 415
	\pinlabel {\red $a_0$} at 890 440
	\pinlabel {\red $a_2$} at 1450 440
	\pinlabel {\red $a_1$} at 1120 100
	
	\pinlabel {\red $a_0$} at 2050 420
	\pinlabel {\red $a_1$} at 1825 520
	\pinlabel {\red $a_2$} at 2300 255
	\pinlabel {\red $a_3$} at 1800 80
    \endlabellist
\includegraphics[scale=0.145]{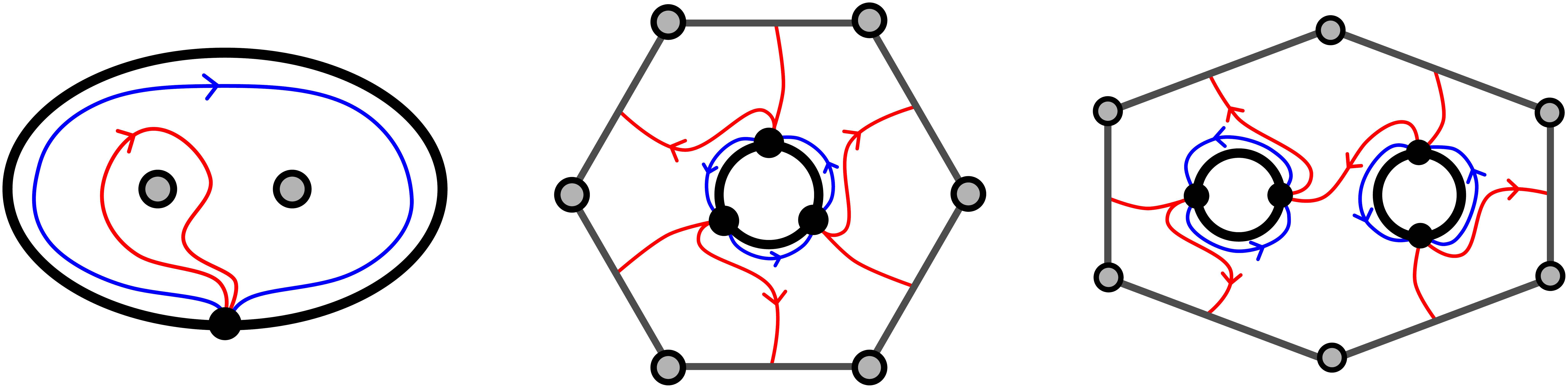}
\end{center}
\caption{(i) Generators of the fundamental group of ${\bf D}_2$. (ii) Generators for the fundamental groupoid $\pi_1({\Sigma_{1}^1}^\circ,\{\tilde x_0,\tilde x_1,\tilde x_2\})$. (iii) Generators for $\pi_1({\Sigma_1^2}^{\circ},\{\tilde x_0,\tilde x_1,\tilde x_2,\tilde x_3\})$.}
\label{Generators}
\end{figure}
\begin{lem}\label{OddBurau}
Given the Burau cover $p_{2g+1} : \Sigma_g^1 \to \Sigma^1_0$ the half twist $H \in \Mod(\Sigma_0^1, \vB)$ lifts to the mapping class $N \in \Mod(\Sigma_g^1)$.
\end{lem}

\begin{proof}
Let $\vG = \pi_1({\bf D}_2,x)$ and let $\vK$ be the fundamental groupoid \linebreak $\pi_1({\Sigma_g^1}^\circ, \{\wt x_i\})$. We will abuse notation by writing $N$ for its image in $\Aut(\vK)$ and $H$ for its image in $\Aut(\vG)$ under the injective natural homomorphisms. We need to show that $N \in \SAut(\vK)$ as defined in Section \ref{fundamental_groupoid}. Since the deck group $D$ is generated by $\iota$, this is equivalent to showing that $N \iota = \iota N$ as automorphisms of $\vK$. It can be seen from Figure \ref{Notch} that $N (a_i) = b_i a_{i+1+g}^{-1}$ and that $N (b_i) = b_i$.

It follows then that
\[
N \iota (a_i) = N (a_{i+1}) = b_{i+1} a_{i+2+g}^{-1} = \iota (b_i a_{i+1+g}^{-1}) = \iota N (a_i), \eand
\]
\[
N \iota (b_i) = N (b_{i+1}) = b_{i+1} = \iota (b_i) = \iota N (b_i).
\]
Since the set $\{a_i, b_i\}$ generates the fundamental groupoid, we are done.

We will now show that the image of $N$ in $\LAut_{\vH}(\vG)$ under the isomorphism $\Pi$ of Lemma \ref{groupoid_birman_hilden} is equal to $H$. This makes sense since $\LMod(\Sigma_0^1, \vB) = \Mod(\Sigma_0^1, \vB)$ and so from Theorem \ref{lmodmain}, we conclude that  $H \in \LAut_{\vH}(\vG)$. We now have
\[
\Pi (N) (a) = p_* N (a_i) = p_* (b_i a_{i+1+g}^{-1}) = b a^{-1} = H(a), \eand
\]
\[
\Pi (N) (b) = p_* N (b_i) = p_* (b_i) = b = H(b).
\]
So $\Pi ( N ) = H$ and since the diagram
\[
\begin{tikzcd}
    \SMod(\Sigma^1_g) \arrow[hookrightarrow]{r}{\wt\Psi} \arrow{d}{\Pi}[swap]{\cong} & \SAut(\vK) \arrow{d}{\Pi}[swap]{\cong} \\
    \Mod(\Sigma^1_0,\vB) \arrow[hookrightarrow]{r}{\Psi} & \LAut_{\vH}(\vG)
\end{tikzcd}
\]
from Lemma \ref{commutative_groupoid_diagram} commutes, the mapping class $N$ is indeed the lift of the half twist $H$.
\end{proof}

\subsection{Even Burau}\label{Even}
We will now move on to the braid group embedding $\beta_{2g+2} : B_2 \hookrightarrow \Mod(\Sigma^2_g)$ given above. As in the odd case we will define an element of $\Mod(\Sigma^2_g)$ and then prove that it is the lift of a half twist $H$. We take $H$ to be the half twist such that $H_*(a)=ba^{-1}$ for $a,b \in \pi_1({\bf D}_2,x)$ as before. We want to find a polygonal representation of $\Sigma_g^2$. We take a regular $(4g+2)$-gon with opposite sides identified. This time, we remove two open disks as shown in Figure \ref{Deck} and label the representation $P$.

\begin{figure}[t]
\begin{center}
\labellist  \hair 1pt
	\pinlabel {$H$} at 290 110
	\pinlabel {$H$} at 985 110
	\pinlabel {$N$} at 290 370
	\pinlabel {$N$} at 985 360
	\endlabellist
\includegraphics[scale=0.266]{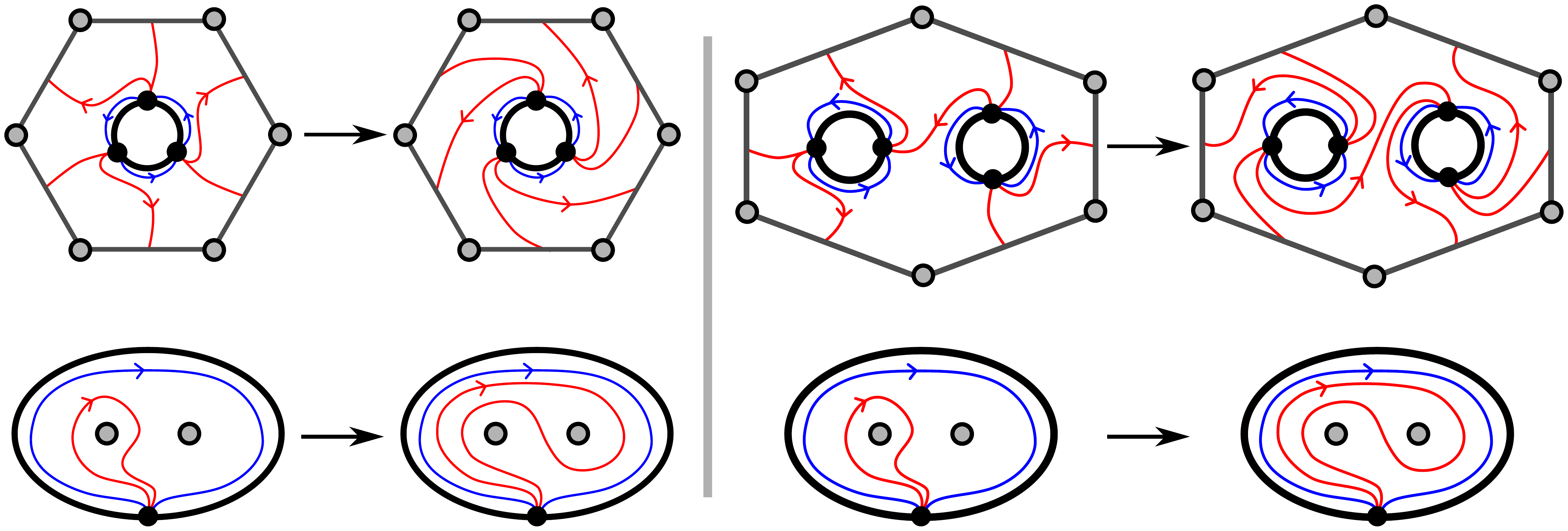}
\end{center}
\caption{The mapping class $H \in \Mod({\bf D}_2)$ and its lifts in the 3 and 4-sheeted Burau covers.  Lemma \ref{OddBurau} shows that $N$ is the lift of $H$.}
\label{Notch}
\end{figure}
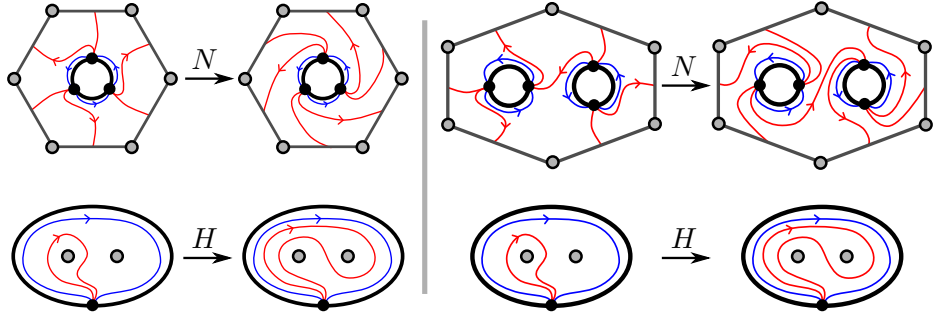

We define an order $2g+2$ homeomorphism $\iota$ as follows:
\begin{enumerate}
\item Cut $P$ along a straight line connecting the top and bottom vertices and label the resulting $(2g+2)$-gons $P_L$ and $P_R$.
\item Rotate both $P_L$ and $P_R$ anti-clockwise by $2\pi / (2g+2)$ and reattach them along the straight line connecting the top and bottom vertices.
\item Rotate $P$ by $\pi$.
\end{enumerate}
For a picture of $\iota$ when $g=1$, see Figure \ref{Deck}(ii).  While this homeomorphism of $\Sigma_g^2$ is substantially more complicated than the one described in Section \ref{Odd} it shares many properties. Both vertices of $P$ are fixed by $\iota$ however, this time, locally $\iota$ is a clockwise rotation by $2\pi / (2g+2)$. It follows that the quotient space $\Sigma_g^2 / \langle \iota \rangle$ is homeomorphic to $\Sigma_0^1$ and the associated covering space is the $(2g+2)$-sheeted Burau cover of $\Sigma_0^1$ with deck group $D \cong \mathbb{Z} / (2g+2) \mathbb{Z}$.

\subsection*{The even notch}
We define $N \in \Mod(\Sigma_g^2)$ to be the mapping class represented by the homeomorphism that rotates the edges of both $P_L$ and $P_R$ by $2 \pi / (2g+2)$ and fixes the boundary components.  See Figure \ref{Notch} for an image of $N$ when $g = 1$.

Using the same method as the proof of Lemma \ref{OddBurau} we arrive at the following result.

\begin{lem}\label{EvenBurau}
Given the Burau cover $p_{2g+2} : \Sigma_g^2 \to \Sigma^1_0$ the half twist $H \in \Mod(\Sigma_0^1, \vB)$ lifts to the mapping class $N \in \Mod(\Sigma_g^1)$.
\end{lem}

The proof of Lemma \ref{EvenBurau} is identical to the odd case, except that while $N(b_i) = b_i$ as before, we now have $N(a_i) = b_i a_{i+1}^{-1}$.

\subsection{Chain twists}\label{Section_Chain_Twists}

We will now describe the two maps defined in the previous section as products of Dehn twists. We will often abuse notation by referring to an isotopy class of curves by the name of a single representative curve.

\subsection*{Chains}
Recall that a sequence of curves $\{c_1, c_2, \dots, c_k\}$ is called a $k$-chain if $i(c_i, c_j) = 1$ if $j = i \pm 1$ and $i(c_i, c_j) = 0$ otherwise. If $k = 2g$ for some $g$ then the closed neighbourhood of $\cup c_i$ is a subsurface homeomorphic to $\Sigma_g^1$ with boundary component isotopic to the curve $d$. Furthermore, if $k = 2g + 1$ then the closed neighbourhood of $\cup c_i$ is a subsurface homeomorphic to $\Sigma_g^2$ with boundary components $d_1$ and $d_2$ (see Figure \ref{BoundChain}).
\begin{figure}[t]
\begin{center}
\labellist \small \hair 1pt
	\pinlabel {$d$} at 545 250
	\pinlabel {$d_1$} at 1540 245
	\pinlabel {$d_2$} at 1540 5
    \endlabellist
\includegraphics[scale=0.21]{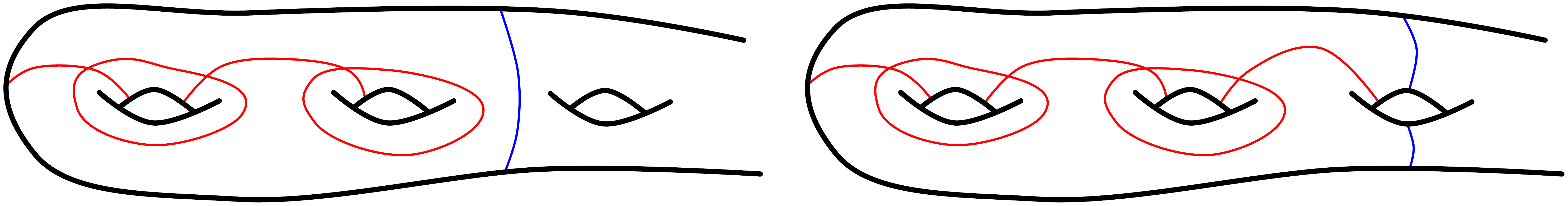}
\end{center}
\caption{A $4$-chain and a $5$-chain. The $4$-chain has a single boundary curve and the $5$-chain has two boundary curves.}
\label{BoundChain}
\end{figure}

By considering the braid group embedding $\beta_2 : B_{k+1} \hookrightarrow \wt \Sigma$ it can be shown that
\[
(T_{c_1} T_{c_2} \dots T_{c_{2g}})^{4g+2} = T_d \eand (T_{c_1} T_{c_2} \dots T_{c_{2g+1}})^{2g+2} = T_{d_1} T_{d_2},
\]
see Farb-Margalit \cite[Section 4.4]{Primer} for more details. Given a $k$-chain $\vC = \{c_1, c_2, \dots, c_k\}$, we call the product $T_\vC := T_{c_1} T_{c_2} \dots T_{c_k}$ a {\it $k$-chain twist} (or a {\it chain twist}). Now, let $p_{2g+1} : \Sigma_g^1 \rightarrow \Sigma_0^1$ be a Burau cover and let $N$ be the lift of the half twist as discussed in Lemma \ref{OddBurau}. If the curve $d$ is isotopic to the boundary of $\Sigma_g^1$ then it is clear that
\[
N^{4g+2} = T_d = (T_\vC)^{4g+2}
\]
for any $2g$-chain $\vC$ in $\Sigma^1_g$. Similarly, suppose $p_{2g+2} : \Sigma_g^2 \rightarrow \Sigma_0^1$ is a Burau cover and $N$ is the lift of the half twist discussed in Lemma \ref{EvenBurau}. If the curves $d_1, d_2$ are isotopic to the boundary components of $\Sigma_g^2$ then we have
\[
N^{2g+2} = T_{d_1} T_{d_2} = (T_\vC)^{2g+2}
\]
for any $(2g+1)$-chain $\vC$ in $\Sigma_g^2$. In Proposition \ref{ChainTwist} we will prove that as well as having the same power as a chain twist, the notch $N$ is in fact equal to a chain twist in both the odd and even cases, proving Theorem \ref{BETAK}. This implies that there exist chains $\vA, \vB$ (of any length) whose corresponding chain twists $T_\vA, T_\vB$ satisfy the braid relation.

In Section \ref{IntersectionData} we will give the explicit combinatorial data required for two $k$-chains to admit chain twists satisfying the braid relation. Furthermore, we show that this data encodes the braid relation on the level of Dehn twists.
\begin{figure}[t]
\begin{center}
\labellist \hair 1pt
    \pinlabel {(i)} at 370 -55
	\pinlabel {(ii)} at 1475 -55
	\pinlabel {$c_1$} at 100 70
	\pinlabel {$c_2$} at 640 70
	\pinlabel {$c_1$} at 1220 565
	\pinlabel {$c_2$} at 1200 70
    \pinlabel {$c_3$} at 1770 545
    \endlabellist
\includegraphics[scale=0.13]{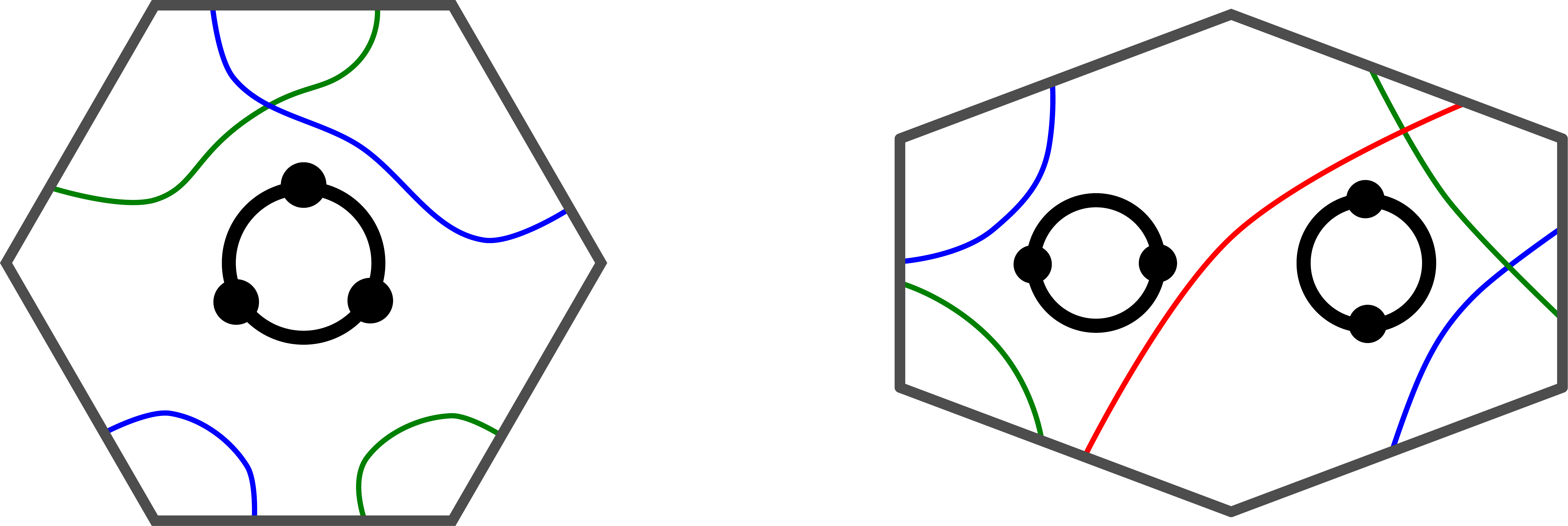}
\end{center}
\caption{The $2$-chain and $3$-chain shown have corresponding chain twists equal to the notch $N$ coming from the $3$-fold and $4$-fold Burau covers respectively.}
\label{2and3chain}
\end{figure}
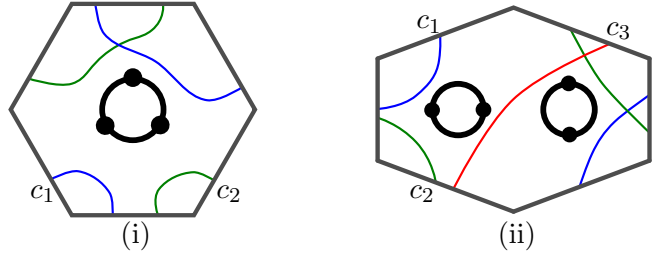
\begin{prop}\label{ChainTwist}
Given a Burau cover $p_k : \wt \Sigma \to \Sigma_0^1$ the half twist $H \in \Mod(\Sigma_0^1, \vB)$ lifts to a $(k-1)$-chain twist.
\end{prop}

\begin{proof}
Given Lemmas \ref{OddBurau} and \ref{EvenBurau} we need only show that the mapping class $N \in \Mod(\Sigma_g^m)$ is equal to a $(k-1)$-chain twist $T_\vC$, for some $(k-1)$-chain $\vC$. To do this we will show that the images of $T_\vC$ and $N$ are equal in the group $\Aut(\vK)$ where $\vK$ is a fundamental groupoid of $\Sigma_g^m$ with basepoints on all boundary components.

To that end, let $k = 2g+1$ and let $\vK$ be the fundamental groupoid of $\Sigma_g^1$ generated by the set depicted in Figure \ref{Generators}(ii). Note that this set was also used to define a fundamental groupoid of ${\Sigma_g^1}^\circ$. In this setting however, the vertices of the polygon are not punctures. Furthermore, in order to facilitate the proof we change the indexing so that $\alpha_{2i} := a_i$. We define the curve $c_0$ uniquely by the groupoid element $\alpha_0 \alpha_1 b_0^{-1}$. We then define $c_i$ to be $N^i(c_0)$ for all $i \in \ZZ / (2g+1)\ZZ$, see Figure \ref{2and3chain}(i). Now, we define the $2g$-chain $\vC : = \{ c_1, \dots, c_{2g}\}$. In fact, we may choose $\vC$ to be any $2g$-chain consisting of the curves $c_i$. Now, by construction it can be seen that
\begin{align*}
T_{c_i}(\alpha_i) &= N (\alpha_i), \\ 
T_{c_j}(\alpha_i) &= \alpha_i \mbox{ for }j > i, \\ 
T_{c_j}N(\alpha_i) &= N(\alpha_i) \mbox{ for }j < i.
\end{align*}
It follows then that for any $i \in \{1,\dots,2g\}$ we have
\begin{align*}
T_\vC(\alpha_i) &= T_{c_1}T_{c_2}\dots T_{c_{2g}}(\alpha_i) \\ 
&= T_{c_1}T_{c_2}\dots T_{c_i}(\alpha_i) \\ 
&= T_{c_1}T_{c_2}\dots T_{c_{i-1}}N(\alpha_i) \\ 
&= N(\alpha_i).
\end{align*}

It remains to show that $T_\vC(\alpha_0) = N(\alpha_0)$. The curve $c_{2g}$ intersects the representative of $\alpha_0$ once and by definition $i(c_{i},c_{i+1})=1$. It therefore follows that the product $T_\vC$ adds a copy of each of the $c_i$ to $\alpha_0$. It is shown in Figure \ref{ChainNotch} that this is in fact equal to $N(\alpha_0)$ in the case where $g=1$, and indeed, this is true for any $g > 0$.
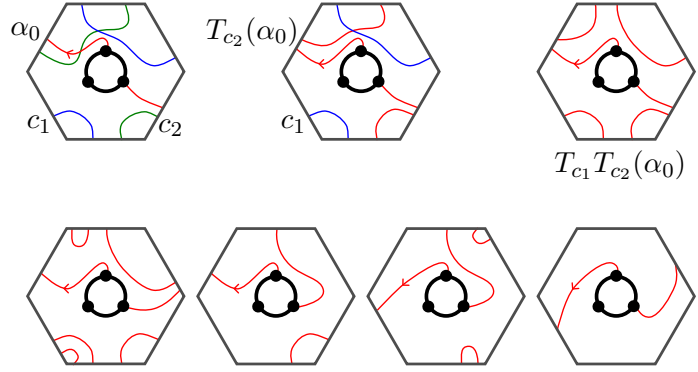
\begin{figure}[t]
\begin{center}
\labellist \hair 1pt
	\pinlabel {$c_1$} at 18 325
	\pinlabel {$c_2$} at 195 320
    \pinlabel {$\alpha_0$} at 0 450
    \pinlabel {$c_1$} at 360 325
    \pinlabel {$T_{c_2}(\alpha_0)$} at 305 450
    \pinlabel {$T_{c_1}T_{c_2}(\alpha_0)$} at 800 270
    \endlabellist
\includegraphics[scale=0.29]{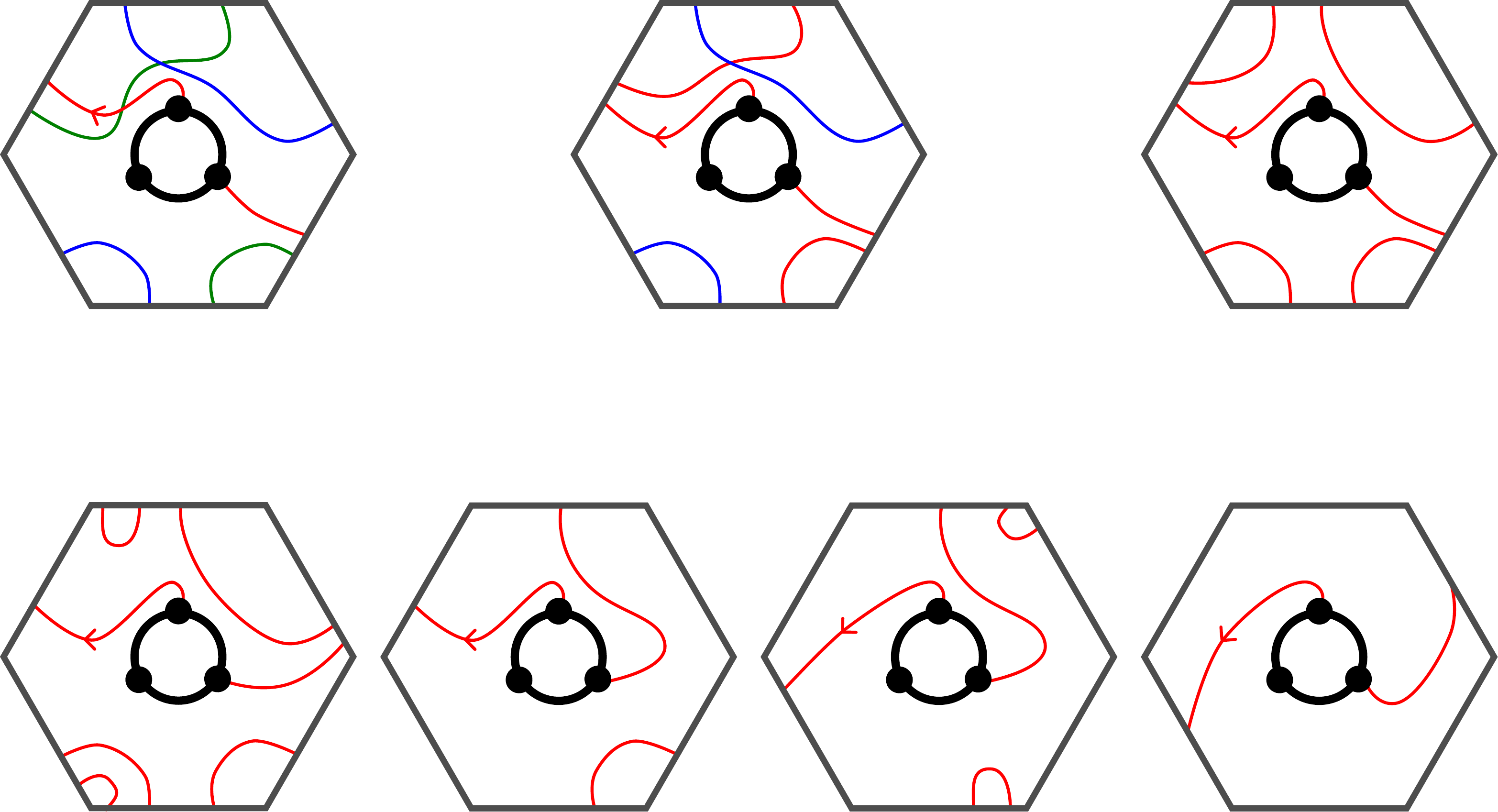}
\end{center}
\caption{All arcs on the bottom row of hexagons are isotopic to the arc shown in the hexagon on the top right. The arcs $T_\vC(\alpha_0)$ and $N(\alpha_0)$ are isotopic and so are equal as elements of the fundamental groupoid.}
\label{ChainNotch}
\end{figure}
It follows that $T_\vC = N$ as groupoid automorphisms, and hence, they are equal as elements of $\Mod(\Sigma_g^1)$.
The case where $k = 2g+2$ is similar to that of $2g+1$. We can use the groupoid generators shown in Figure \ref{Generators}(iii) and the details are left to the reader.
\end{proof}

Note that Proposition \ref{ChainTwist} proves Theorem \ref{BETAK} in that it completely determines the image of each braid group generator $\sigma_i$ by the homomorphism $\beta_k$.

\subsection{Intersection data}\label{IntersectionData}

In Section \ref{Section_Chain_Twists} we saw that half twists lift to $(k-1)$-chain twists with respect to any $k$-sheeted Burau cover where $k\ge3$. We will now explicitly describe a sufficient combinatorial condition for two chains that implies their chain twists satisfy the braid relation.  We will assume that $k \geq 3$ for the remainder of this section.

\subsection*{Bracelets}

Let $\vC = \{c_1,\ldots,c_{k-1}\}$ be a $(k-1)$-chain and let $c_0 = T_{\vC}(c_{k-1})$.  We call the set $\{c_i:i \in \ZZ/k\ZZ\}$ a {\it $k$-bracelet} (or a {\it bracelet}). Such a bracelet is called the {\it bracelet completion} of the chain $\vC$.

\begin{lem}\label{bracelet_properties}
Let $\vC = \{c_1,\ldots,c_{k-1}\}$ be a chain and $\{c_i:i \in \ZZ/k\ZZ\}$ the bracelet completion of $\vC$.  Then
\begin{enumerate}
\item[{\normalfont (i)}] $i(c_i,c_j) = \begin{cases}
1 &\text{if } i = j-1,j+1 \\
0 &\text{otherwise,}
\end{cases}$
\item[{\normalfont (ii)}] $T_{c_i}T_{c_{i+1}} \cdots T_{c_{i-2}} = T_{c_{i+1}}T_{c_{i+2}} \cdots T_{c_{i-1}}$ for all $i \in \ZZ/k\ZZ$.
\end{enumerate}
\end{lem}
\begin{proof}
Note that $T_{c_0} = T_{c_1} \cdots T_{c_{k-1}}T_{c_{k-2}}^{-1} \cdots T_{c_1}^{-1}$.  Any solution to the word problem for the braid group (Dehornoy's handle reduction \cite{Dehornoy} for example) can be used to show $[T_{c_0},T_{c_i}] = 1$ if $i \neq 1, k-1$ and $T_{c_0}T_{c_i}T_{c_0} = T_{c_i}T_{c_0}T_{c_i}$ if $i = 1,k-1$.  This proves property (i).

For (ii), note $T_{c_0}\cdots T_{c_{k-2}} = T_{c_1}\cdots T_{c_{k-1}}$ from the definition of $c_0$.  Suppose now $T_{c_j} \cdots T_{c_{j-2}} = T_{c_{j+1}} \cdots T_{c_{j-1}}$ for some $j\in \ZZ/k\ZZ$.  Then
\begin{align*}
T_{c_{j+2}} \cdots T_{c_{j}} &= T_{c_{j+1}}^{-1}T_{c_{j+1}}T_{c_{j+2}}\cdots T_{c_{j-1}}T_{c_j} \\
&= T_{c_{j+1}}^{-1}T_{c_j}T_{c_{j+1}}\cdots T_{c_{j-2}}T_{c_{j}} \\
&= T_{c_j}T_{c_{j+1}}T_{c_j}^{-1}T_{c_{j+2}} \cdots T_{c_{j-2}}T_{c_j} \\
&= T_{c_j} \cdots T_{c_{j-2}}
\end{align*}
completing the proof.
\end{proof}

Note that property (ii) in Lemma \ref{bracelet_properties} implies that a $k$-bracelet is the completion of any of the $(k-1)$-chains obtained by deleting a curve.  Abusing notation, suppose $\vC$ is a $k$-bracelet.  In light of this fact, define the {\it bracelet twist} $T_{\vC}$ as the chain twist about any of the $(k-1)$-chains obtained by deleting a curve from $\vC$.

\subsection*{Mesh intersection}
Let $\vA = \{a_i:i \in \ZZ/k\ZZ\}$ and $\vB = \{b_j:j \in \ZZ/k\ZZ\}$ be two $k$-bracelets.  We say $\vA$ and $\vB$ have {\it mesh intersection} if there exists $t \in \ZZ/k\ZZ$ such that 
\[
i(a_i,b_{j+t}) = \begin{cases}
1 &\text{if } i = j, j+1 \\
0 &\text{otherwise.}
\end{cases}
\]
In practice, we may simply relabel the curves in $\vB$ and assume $t = 0$.

Fix $k \geq 3$ and let $\beta_k:B_3 \to \Mod(\Sigma_g^m)$ be the embedding of the braid group arising from the $k$-sheeted Burau cover.  Proposition \ref{ChainTwist} shows that each standard generator is sent to a $(k-1)$-chain twist.  The proof proceeds by first constructing a set of $k$ curves, and then arbitrarily discarding one.  The set of $k$ curves constructed is in fact the completion of the $(k-1)$-chain.  Furthermore, it can be checked that $\beta_k$ sends the two standard generators to bracelet twists about bracelets with mesh intersection. In fact, in the discussion following the statement of Proposition \ref{ChainBraid} we will see that if $\vA$ and $\vB$ are two bracelets with mesh intersection such that all $2k$ curves are distinct, then there exists a Burau cover that lifts two half twists satisfying a braid relation to $T_\vA$ and $T_\vB$.  It follows from Theorem \ref{ClassyLMod} that $T_\vA$ and $T_\vB$ satisfy the braid relation.  By leveraging the algebraic properties of Dehn twists, we may arrive at the same conclusion without mention of such a covering space.

\begin{figure}[t]
\begin{center}
\labellist\hair 1pt
	\pinlabel {$a_0$} at 645 280
	\pinlabel {$c$} at 290 305
	\pinlabel {$a_2$} at 400 165
	\pinlabel {$b_0$} at 185 305
	\pinlabel {$b_1$} at 330 90
    \endlabellist
\includegraphics[scale=0.25]{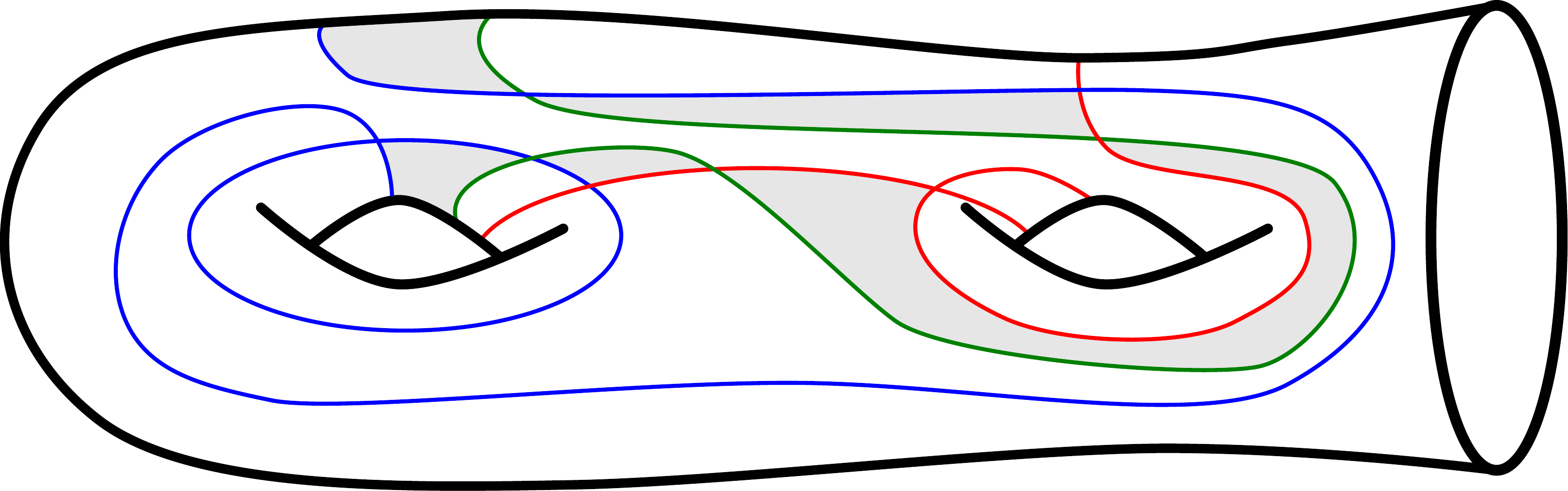}
\end{center}
\caption{Two $3$-bracelets $\vA=\{a_0, c, a_2\}$ and $\vB=\{b_0,b_1,c\}$ with mesh intersection.}
\label{Degenerate}
\end{figure}
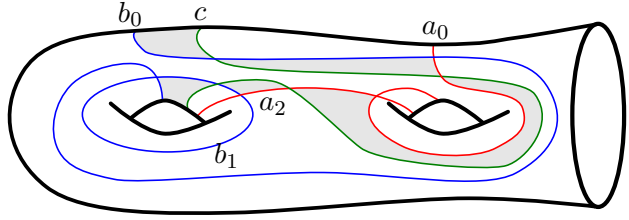

\begin{thm}\label{MeshBraid}
If two $k$-bracelets $\vA$ and $\vB$ have mesh intersection then $T_\vA T_\vB T_\vA = T_\vB T_\vA T_\vB$.
\end{thm}

\begin{proof}
By relabelling the curves in $\vB$, we may assume $t = 0$ in the definition of mesh intersection.  We first show that $T_\vA T_\vB T_{a_i} = T_{b_i} T_\vA T_\vB$ as follows;
\begin{align*}
T_{\vA}T_{\vB}T_{a_i} &= T_{a_{i+1}} \cdots T_{a_{i-1}}T_{b_{i}} \cdots T_{b_{i-2}}T_{a_i} \\
&= T_{a_{i+1}}T_{b_i}T_{a_{i+2}} \cdots T_{a_{i-1}}T_{b_{i+1}} \cdots T_{b_{i-2}}T_{a_i} \\
&= T_{a_{i+1}}T_{b_i}T_{a_{i+2}} \cdots T_{a_{i-1}}T_{a_i}T_{b_{i+1}} \cdots T_{b_{i-2}} \\
&= T_{a_{i+1}}T_{b_i}T_{a_{i+1}} \cdots T_{a_{i-1}}T_{b_{i+1}} \cdots T_{b_{i-2}} \\
&= T_{b_i}T_{a_{i+1}}T_{b_i}T_{a_{i+2}} \cdots T_{a_{i-1}}T_{b_{i+1}} \cdots T_{b_{i-2}} \\
&= T_{b_i}T_{a_{i+1}}T_{a_{i+2}} \cdots T_{a_{i-1}}T_{b_i}T_{b_{i+1}} \cdots T_{b_{i-2}} \\
&= T_{b_i}T_{\vA}T_{\vB}.
\end{align*}
The second and third equalities come from the intersection data of curves $b_i$ and $a_{i-1}$ respectively. The fourth equality comes from property (ii) of Lemma \ref{bracelet_properties}. The fifth and sixth equalities come from property (i) of Lemma \ref{bracelet_properties} applied to $\vB$.

This allows us to achieve the braid relation as follows:
\begin{align*}
T_\vA T_\vB T_\vA &= T_\vA T_\vB T_{a_0}T_{a_1} \cdots T_{a_{k-2}}
\\ &= T_{b_0} T_\vA T_\vB T_{a_1} \cdots T_{a_{k-2}}
\\ &= \cdots
\\ &= T_{b_0} T_{b_1} \cdots T_{b_{k-2}} T_\vA T_\vB
\\ &= T_\vB T_\vA T_\vB. \hfill \qedhere
\end{align*}
\end{proof}

When $k \geq 4$, it can be shown that if two $k$-bracelets have mesh intersection, then all $2k$ curves in question are distinct.  However, this is not the case when $k = 3$.  Figure \ref{Degenerate} shows two $3$-bracelets $\{a_i:i \in \ZZ/3\ZZ\}$ and $\{b_i:i \in \ZZ/3\ZZ\}$ with mesh intersection such that $a_1 = b_2$.

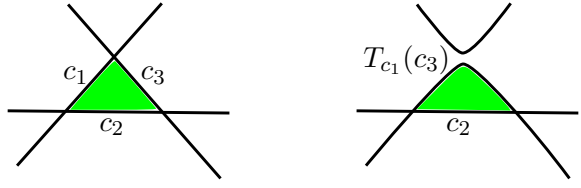
\begin{figure}[t]
\begin{center}
\labellist\hair 1pt
	\pinlabel {$c_1$} at 80 115
	\pinlabel {$c_2$} at 115 65
	\pinlabel {$c_3$} at 155 115
	\pinlabel {$T_{c_1}(c_3)$} at 400 130
	\pinlabel {$c_2$} at 450 65
    \endlabellist
\includegraphics[scale=.4]{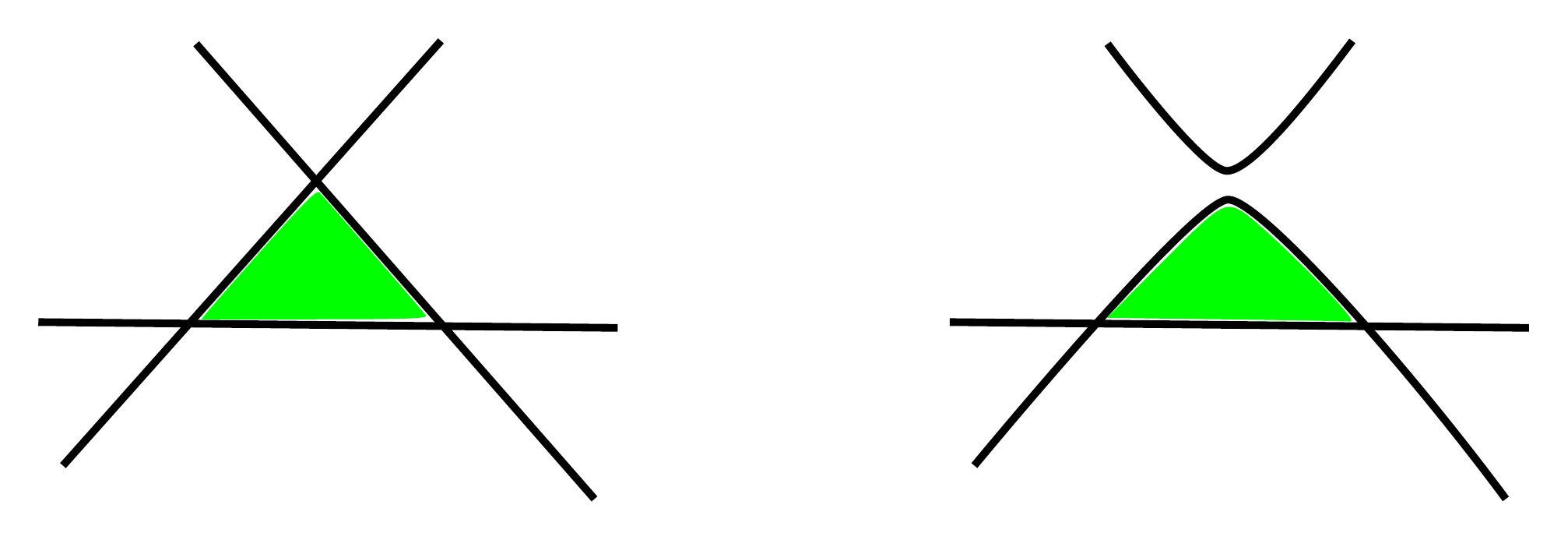}
\end{center}
\caption{The triple $(c_1,c_2,c_3)$ bounds a positively oriented triangle.  The right image shows the bigon between $T_{c_1}(c_3)$ and $c_2$, implying $i(T_{c_1}(c_3),c_2) = 0$.}
\label{triangles}
\end{figure}

\subsection*{Intersection data for chains}
We now shift our attention to finding a sufficient combinatorial condition for two $(k-1)$-chains $\vA$ and $\vB$ to have the property that their bracelet completions have mesh intersection.  We will then be able to conclude, by Theorem \ref{MeshBraid}, that the two chain twists $T_{\vA}$ and $T_{\vB}$ satisfy a braid relation.

Suppose $\alpha_1,\alpha_2,\alpha_3$ are curves on a surface in minimal position such that $i(\alpha_i,\alpha_j) = 1$ if $i \neq j$.  The graph given by the three curves defines two triangles and a hexagon on the surface.  Suppose one of the triangles bounds a disk $\bf D$.  We say the triple $(\alpha_1,\alpha_2,\alpha_3)$ {\it bounds a positively oriented triangle} if you can traverse $\partial \bf D$ in a anti-clockwise direction from the intersection point $x \in \alpha_1 \cap \alpha_3$ and travel along a segment of $\alpha_1$, then a segment of $\alpha_2$, then a segment of $\alpha_3$ in that order and return to $x$.  See Figure \ref{triangles} for a local picture of three curves bounding a positively oriented triangle.

We say a triple $(c_1,c_2,c_3)$ of isotopy classes of curves bounds a positively oriented triangle if there exist representatives $\gamma_i$ of $c_i$ such that $(\gamma_1,\gamma_2,\gamma_3)$ bounds a positively oriented triangle.

Note that if $(c_1,c_2,c_3)$ bounds a positively oriented triangle and $\sigma \in S_3$ is a permutation, then $(c_{\sigma(1)},c_{\sigma(2)},c_{\sigma(3)})$ bounds a positively oriented triangle if and only if $\sigma$ is an even permutation.

The importance of the definition of a triple bounding a positively oriented triangle is that if $(c_1,c_2,c_3)$ bounds a positively oriented triangle, then $i(T_{c_1}(c_3),c_2) = 0$.  This can be seen in Figure \ref{triangles}.

\begin{prop}\label{ChainBraid}
Suppose $\vA = \{a_1,\ldots,a_{k-1}\}$ and $\vB = \{b_1,\ldots,b_{k-1}\}$ are two $(k-1)$-chains with the property that
\begin{enumerate}
\item[\normalfont (i)] $i(a_i,b_j) = \begin{cases}
1 &\text{if } i = j,j+1, \\
0 &\text{otherwise,}
\end{cases}$ 
\item[\normalfont (ii)] The triples $(a_i,b_i,a_{i+1})$ and $(b_i,a_{i+1},b_{i+1})$ bound positively oriented triangles for all $i \in \{1,\ldots,k-2\}$.
\end{enumerate}
Then the chain twists $T_{\vA}$ and $T_{\vB}$ satisfy a braid relation.
\end{prop}

If each of the $2k-2$ curves in $\vA$ and $\vB$ are distinct then we may view them as depicted in Figure \ref{Mesh}.

Defining $\Sigma_g^m$ to be the regular neighbourhood of the curves and triangles it can be seen that $m = \gcd(3,k)$.  Furthermore, an Euler characteristic argument shows that $g = k-2$ if $m=3$ and $g=k-1$ if $m=1$.  These are precisely the values of $m$ and $g$ that give rise to a $k$-sheeted Burau cover $p_k : \Sigma_g^m \to \Sigma_0^1$ with three branch points.  By using a variation of the change of coordinates principle (see \cite[Section 1.3.2]{Primer}) we may conclude that $T_\vA$ and $T_\vB$ are lifts of half twists that satisfy a braid relation. Hence from Theorem \ref{ClassyLMod} we conclude that $T_\vA$ and $T_\vB$ satisfy a braid relation.

Unlike the discussion above, the following proof of Proposition \ref{ChainBraid} does not make use of the Birman-Hilden Theorem.  As such, it provides a more intrinsic perspective of chain twists satisfying a braid relation, and deals with the case when the curves in $\vA$ and $\vB$ are not distinct.
\begin{figure}[t]
\centering
\includegraphics[scale=0.3]{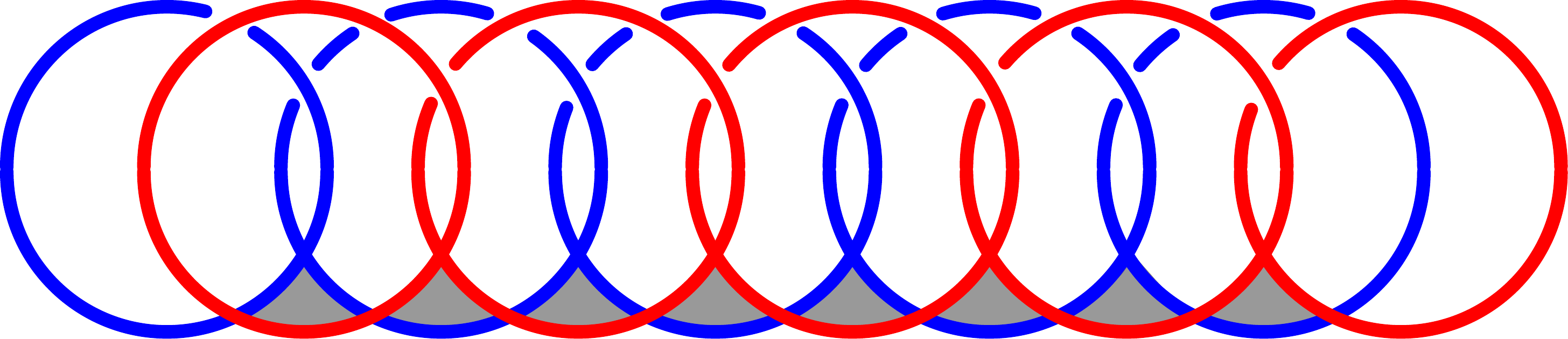}
\caption{Intersecting circles describing two $5$-chains satisfying the conditions of Proposition \ref{ChainBraid}. The regular neighbourhood of this collection of curves and triangles is homeomorphic to $\Sigma^3_4$.}
\label{Mesh}
\end{figure}
\begin{proof}[Proof of Proposition \ref{ChainBraid}]
Let $a_0 = T_{\vA}(a_{k-1})$ and $b_0 = T_{\vB}(b_{k-1})$.  To ease notation let $\Delta = T_{b_1}\cdots T_{b_{k-2}}$ and $\nabla = T_{b_2}\cdots T_{b_{k-1}}$.  Note that $T_{b_0} = \Delta T_{b_{k-1}} \Delta^{-1} = \nabla^{-1}T_{b_1}\nabla$.

By Theorem \ref{MeshBraid} it suffices to show 
\[
i(a_0,b_j) = \begin{cases}
1 &\text{if } j = 0,k-1 \\
0 &\text{otherwise}
\end{cases}
\eand
i(a_i,b_0) = \begin{cases}
1 &\text{if }i = 0,1 \\
0 &\text{otherwise.}
\end{cases}
\]
Let $i \neq 0,1$.  Since $(b_j,a_{j+1},b_{j+1})$ bounds a positively triangle for all $j \in\{1,\ldots,k-2\}$, we have $i(T_{b_j}(b_{j+1}),a_{j+1}) = 0$ so $[T_{b_j}T_{b_{j+1}}T_{b_j}^{-1},T_{a_{j+1}}] = 1$.  Rearranging and relabelling we get
\[
T_{b_{i}}^{-1}T_{b_{i-1}}^{-1}T_{a_i}T_{b_{i-1}}T_{b_i} = T_{b_{i-1}}^{-1}T_{a_i}T_{b_{i-1}}
\]
for all $i \neq 0,1$. We have
\begin{align*}
[T_{b_0},T_{a_i}] &= \Delta T_{b_{k-1}}\Delta^{-1}T_{a_i}\Delta T_{b_{k-1}}^{-1}\Delta^{-1}T_{a_i}^{-1} \\
&= \Delta T_{b_{k-1}}T_{b_{k-2}}^{-1} \cdots T_{b_{i+1}}^{-1}T_{b_i}^{-1}T_{b_{i-1}}^{-1}T_{a_i}\\
& \quad \quad \quad T_{b_{i-1}}T_{b_i}T_{b_{i+1}} \cdots T_{b_{k-2}}T_{b_{k-1}}^{-1}\Delta^{-1}T_{a_i}^{-1} \\
&= \Delta T_{b_{k-1}}T_{b_{k-2}}^{-1} \cdots T_{b_{i+1}}^{-1}T_{b_{i-1}}^{-1}T_{a_i}T_{b_{i-1}}T_{b_{i+1}} \cdots T_{b_{k-2}}T_{b_{k-1}}^{-1}\Delta^{-1}T_{a_i}^{-1} \\
&= T_{b_1}\cdots T_{b_{i-1}}T_{b_i}T_{b_{i-1}}^{-1}T_{a_i}T_{b_{i-1}}T_{b_i}^{-1}T_{b_{i-1}}^{-1} \cdots T_{b_1}^{-1}T_{a_i}^{-1} \\
&= T_{b_1} \cdots T_{b_{i-1}}T_{b_{i-1}}^{-1}T_{a_i}T_{b_{i-1}}T_{b_{i-1}}^{-1}\cdots T_{b_1}^{-1}T_{a_i}^{-1} \\
&= T_{b_1} \cdots T_{b_{i-2}}T_{a_i}T_{b_{i-2}}^{-1} \cdots T_{b_1}^{-1}T_{a_i}^{-1} \\
&= 1.
\end{align*}
Therefore $i(a_i,b_0) = 0$.  When $i = 1$ we have
\begin{align*}
T_{a_1}T_{b_0}T_{a_1}&= T_{a_1}\nabla^{-1}T_{b_1}\nabla T_{a_1}\\
&= \nabla^{-1}T_{a_1}T_{b_1}T_{a_1}\nabla \\
&= \nabla^{-1}T_{b_1}T_{a_1}T_{b_1}\nabla \\
&= \nabla^{-1}T_{b_1}\nabla T_{a_1}\nabla^{-1}T_{b_1}\nabla \\
&= T_{b_0}T_{a_1}T_{b_0}
\end{align*}
so $i(a_1,b_0) = 1$.  Similar arguments show $i(a_0,b_j) = 0$ for $j \neq 0,k-1$ and $i(a_0,b_{k-1}) = 1$.

It remains to show $i(a_0,b_0) = 1$.  We have
\begin{align*}
T_{a_0}T_{b_0}T_{a_0} &= T_{a_0}\Delta T_{b_{k-1}}\Delta^{-1} T_{a_0}\\
&= \Delta T_{a_0}T_{b_{k-1}}T_{a_0}\Delta^{-1} \\
&= \Delta T_{b_{k-1}}T_{a_0}T_{b_{k-1}}\Delta^{-1} \\
&= \Delta T_{b_{k-1}}\Delta^{-1}T_{a_0}\Delta T_{b_{k-1}}\Delta^{-1} \\
&= T_{b_0}T_{a_0}T_{b_0}
\end{align*}
completing the proof.
\end{proof}

See Figure \ref{mesh_on_a_surface} for two 3-chains on $\Sigma_3^1$ satisfying the conditions of Lemma \ref{ChainBraid}.  The positively oriented triangles are shaded in grey.

\subsection{Open Questions}
Here are a few natural questions relating to the braid group embeddings constructed above.  Recall that for each $k \geq 3$ and $n \geq 2$ we have constructed an embedding $\beta_k: B_n \hookrightarrow \Mod(\Sigma_g^m)$ arising from the $k$-sheeted Burau cover.  Here, $m = \gcd(n,k)$ and $g = 1-\frac 12(k+n+m-nk)$.

\subsection*{Necessity of mesh intersection}

When two simple closed curves $a$ and $b$ on a surface intersect once, then $T_a$ and $T_b$ satisfy a braid relation.  In fact, this condition is necessary.  That is, $T_aT_bT_a = T_bT_aT_b$ if and only if $i(a,b) = 1$ (see \cite[\S 3.5]{Primer}).  The next question asks the analogous question for chain twists.

\begin{qu}
Suppose $\vA$ and $\vB$ are $k$-chains for $k \geq 2$, and let $T_{\vA}$ and $T_{\vB}$ be the corresponding chain twists.  Is it true that if $T_{\vA}T_{\vB}T_{\vA} = T_{\vB}T_{\vA}T_{\vB}$, then the bracelet completions of $\vA$ and $\vB$ have mesh intersection?
\end{qu}

\subsection*{Automorphisms of free groups}
For a surface $\Sigma$ with non-empty boundary, there is a homomorphism $\Mod(\Sigma) \to \Aut(\pi_1(\Sigma))$ given by the action of $\Mod(\Sigma)$ on the fundamental group of $\Sigma$ with a basepoint on the boundary.  For a surface of genus $g$ and $m$ boundary components, $\pi_1(\Sigma_g^m) \cong F_{2g+m-1}$.  Precomposing with the braid group embeddings above, we get an induced homomorphism from the braid group into the automorphism group of a free group.

\begin{qu}
Let $k\geq 3$.  For each $n \geq 2$ there is a homomorphism $\phi_{n,k}:B_n \to F_{(n-1)(k-1)}$.  What can be said about this family of homomorphisms?  Do they give rise to new embeddings of the braid group in $\Aut(F_n)$?
\end{qu}

\subsection*{Triviality of the induced map on stable homology}

There is a geometric embedding $B_{2g} \hookrightarrow \Mod(\Sigma_g^1)$ for each $g$.  This family of embeddings gives a map from $B_\infty = \lim_{g \to \infty} B_{2g}$ to $\Gamma_{\infty} = \lim_{g \to \infty} \Mod(\Sigma_g^1)$.  In the 1980s J. Harer conjectured that the induced map on stable homology $H_*(B_\infty;\ZZ/2\ZZ) \to H_*(\Gamma_{\infty};\ZZ/2\ZZ)$ is trivial.  The conjecture was proved by Song and Tillman in \cite[Theorem 1.1]{ST}.  A stronger version of Harer's conjecture was proved for a large family of non-geometric embeddings of the braid group in \cite{BT}.

\begin{qu}
Fix $k > 3$.  Is there a version of Harer's conjecture that is true with respect to the embeddings $\beta_k:B_n \hookrightarrow \Mod(\Sigma_g^m)$?
\end{qu}
Note that this question is answered affirmatively for stable homology with any coefficients when $k = 3$ by Kim-Song \cite[Theorem 3.4]{KS}.

\subsection*{Classifying braid embeddings}
There are now infinite families of non-geometric embeddings of braid groups in mapping class groups.

\begin{qu}
Is there a classification of all possible conjugacy classes of embeddings of the braid group in the mapping class group?
\end{qu}

The proof of Theorem \ref{MeshBraid} suggests a way to construct more examples as follows.

Suppose we have two subsets of mapping classes $\{\phi_i\}$ and $\{\theta_i\}$ indexed by $\ZZ / k\ZZ$ such that
\[
\phi_i \phi_j = \begin{cases} 
\phi_j \phi_i \phi_j \phi_i^{-1} & \text{if } j = i-1,i+1, \\
\phi_j \phi_i & \text{otherwise,}
\end{cases}
\]
\[
\theta_i \theta_j = \begin{cases} 
\theta_j \theta_i \theta_j \theta_i^{-1} & \text{if } j = i-1,i+1, \\
\theta_j \theta_i & \text{otherwise.}
\end{cases}
\]
Suppose further that
\[
\phi_i \theta_j = \begin{cases} 
\theta_j \phi_i \theta_j \phi_i^{-1} & \text{if } i = j,j+1, \\
\theta_j \phi_i & \text{otherwise,}
\end{cases}
\]
and for any $i \in \ZZ/k \ZZ$ we have
\[
\Phi := \phi_i \dots \phi_{i-2} = \phi_{i+1} \dots \phi_{i-1} \eand 
\Theta := \theta_i \dots \theta_{i-2} = \theta_{i+1} \dots \theta_{i-1}.
\]
Then the products $\Phi$ and $\Theta$ satisfy the braid relation, that is, $\Phi \Theta \Phi  = \Theta \Phi \Theta$.

We conjecture however, that this is only possible when each $\phi_i$ and $\theta_i$ is a Dehn twist and the corresponding sets of curves are bracelets with mesh intersection. There is no particular reason to assume otherwise, except to satisfy our own insatiable desire for pattern.

%%%%%%%%%%%%%%%%%%%%%%%%%%%%%%%%%%%

\bibliographystyle{plain}
%\addcontentsline{toc}{section}{Bibliography}
\bibliography{bhboundary}

\begin{thebibliography}{10}

\bibitem{AM}
Murat Alp and Christopher~D. Wensley.
\newblock Automorphisms and homotopies of groupoids and crossed modules.
\newblock {\em Appl. Categ. Structures}, 18(5):473--504, 2010.

\bibitem{ALS}
Javier Aramayona, Christopher~J. Leininger, and Juan Souto.
\newblock Injections of mapping class groups.
\newblock {\em Geom. Topol.}, 13(5):2523--2541, 2009.

\bibitem{BB}
Stephen~J. Bigelow and Ryan~D. Budney.
\newblock The mapping class group of a genus two surface is linear.
\newblock {\em Algebr. Geom. Topol.}, 1:699--708, 2001.

\bibitem{BH}
Joan~S. Birman and Hugh~M. Hilden.
\newblock On isotopies of homeomorphisms of {R}iemann surfaces.
\newblock {\em Ann. of Math. (2)}, 97:424--439, 1973.

\bibitem{BT}
Carl-Friedrich B\"odigheimer and Ulrike Tillmann.
\newblock Embeddings of braid groups into mapping class groups and their
  homology.
\newblock In {\em Configuration spaces}, volume~14 of {\em CRM Series}, pages
  173--191. Ed. Norm., Pisa, 2012.

\bibitem{BMP}
Tara Brendle, Dan Margalit, and Andrew Putman.
\newblock Generators for the hyperelliptic {T}orelli group and the kernel of
  the {B}urau representation at {$t=-1$}.
\newblock {\em Invent. Math.}, 200(1):263--310, 2015.

\bibitem{BM}
Tara~E. Brendle and Dan Margalit.
\newblock The level four braid group.
\newblock {\em J. Reine Angew. Math.}, 735:249--264, 2018.

\bibitem{Brown}
Ronald Brown.
\newblock {\em Topology and groupoids}.
\newblock BookSurge, LLC, Charleston, SC, 2006.
\newblock Third edition of {{\i}t Elements of modern topology} [McGraw-Hill,
  New York, 1968; MR0227979], With 1 CD-ROM (Windows, Macintosh and UNIX).

\bibitem{Dehornoy}
Patrick Dehornoy.
\newblock A fast method for comparing braids.
\newblock {\em Adv. Math.}, 125(2):200--235, 1997.

\bibitem{Endo}
Hisaaki Endo.
\newblock Meyer's signature cocycle and hyperelliptic fibrations.
\newblock {\em Math. Ann.}, 316(2):237--257, 2000.

\bibitem{Primer}
Benson Farb and Dan Margalit.
\newblock {\em A primer on mapping class groups}, volume~49 of {\em Princeton
  Mathematical Series}.
\newblock Princeton University Press, Princeton, NJ, 2012.

\bibitem{GW}
Tyrone Ghaswala and Rebecca~R. Winarski.
\newblock Lifting homeomorphisms and cyclic branched covers of spheres.
\newblock {\em Michigan Math. J.}, 66(4):885--890, 2017.

\bibitem{Higgins}
Philip~J. Higgins.
\newblock {\em Notes on categories and groupoids}.
\newblock Van Nostrand Reinhold Co., London-New York-Melbourne, 1971.
\newblock Van Nostrand Rienhold Mathematical Studies, No. 32.

\bibitem{KK}
Nariya {Kawazumi} and Yusuke {Kuno}.
\newblock Groupoid-theoretical methods in the mapping class groups of surfaces.
\newblock {\em \href{https://arxiv.org/abs/1109.6479}{arXiv:1109.6479v3}},
  2012.

\bibitem{Kerckhoff}
Steven~P. Kerckhoff.
\newblock The {N}ielsen realization problem.
\newblock {\em Ann. of Math. (2)}, 117(2):235--265, 1983.

\bibitem{KS}
Byung~Chun Kim and Yongjin Song.
\newblock Configuration spaces, moduli spaces and 3-fold covering spaces.
\newblock {\em Manuscripta Math.}, 160(3-4):391--409, 2019.

\bibitem{Lefschetz}
Solomon Lefschetz.
\newblock On the fixed point formula.
\newblock {\em Annals of Mathematics}, 38(4):819--822, 1937.

\bibitem{MH}
Colin Maclachlan and William~J. Harvey.
\newblock On mapping-class groups and {T}eichm\"uller spaces.
\newblock {\em Proc. London Math. Soc. (3)}, 30(part 4):496--512, 1975.

\bibitem{MW}
Dan {Margalit} and Rebecca~R. {Winarski}.
\newblock The {B}irman-{H}ilden theory.
\newblock {\em \href{https://arxiv.org/abs/1703.03448}{arXiv:1703.03448v1}},
  2017.

\bibitem{McLeayThesis}
Alan McLeay.
\newblock {\em Subgroups of mapping class groups and braid groups}.
\newblock PhD thesis, University of Glasgow, 2018.

\bibitem{McMullen}
Curtis~T. McMullen.
\newblock Braid groups and {H}odge theory.
\newblock {\em Math. Ann.}, 355(3):893--946, 2013.

\bibitem{Morifuji}
Takayuki Morifuji.
\newblock On {M}eyer's function of hyperelliptic mapping class groups.
\newblock {\em J. Math. Soc. Japan}, 55(1):117--129, 2003.

\bibitem{Song}
Yongjin Song.
\newblock The braidings in the mapping class groups of surfaces.
\newblock {\em J. Korean Math. Soc.}, 50(4):865--877, 2013.

\bibitem{ST}
Yongjin Song and Ulrike Tillmann.
\newblock Braids, mapping class groups, and categorical delooping.
\newblock {\em Math. Ann.}, 339(2):377--393, 2007.

\bibitem{Stukow1}
Micha{\l} Stukow.
\newblock Conjugacy classes of finite subgroups of certain mapping class
  groups.
\newblock {\em Turkish J. Math.}, 28(2):101--110, 2004.

\bibitem{Stukow2}
Micha{\l} Stukow.
\newblock Small torsion generating sets for hyperelliptic mapping class groups.
\newblock {\em Topology Appl.}, 145(1-3):83--90, 2004.

\bibitem{Szepietowski}
B{\l}a\.zej Szepietowski.
\newblock Embedding the braid group in mapping class groups.
\newblock {\em Publ. Mat.}, 54(2):359--368, 2010.

\bibitem{Wajnryb}
Bronislaw Wajnryb.
\newblock Relations in the mapping class group.
\newblock In {\em Problems on mapping class groups and related topics},
  volume~74 of {\em Proc. Sympos. Pure Math.}, pages 115--120. Amer. Math.
  Soc., Providence, RI, 2006.

\bibitem{Winarski}
Rebecca~R. Winarski.
\newblock Symmetry, isotopy, and irregular covers.
\newblock {\em Geom. Dedicata}, 177:213--227, 2015.

\end{thebibliography}

\end{document}